\documentclass[11pt,reqno]{amsart}

\newtheorem{thm}{Theorem}[section]
\newtheorem*{thm*}{Theorem}
\newtheorem{lem}[thm]{Lemma}
\newtheorem*{lem*}{Lemma}

\newtheorem{cor}[thm]{Corollary}

\newtheorem{prop}[thm]{Proposition}

\theoremstyle{definition}

\newtheorem{assump}[thm]{Assumption}
\renewcommand{\thecase}{}
\newtheorem*{case*}{Case}

\newtheorem{defn}[thm]{Definition}
\newtheorem*{defn*}{Definition}

\newtheorem*{exmp*}{Example}

\newtheorem{rmk}[thm]{Remark}
\newtheorem*{rmk*}{Remark}
\newtheorem{step}{Step}\renewcommand{\thestep}{}

\theoremstyle{remark}

\makeatletter
\def\alphenumi{
  \def\theenumi{\alph{enumi}}
  \def\p@enumi{\theenumi}
  \def\labelenumi{(\@alph\c@enumi)}}
\makeatother

\makeatletter
\def\thecase{\@arabic\c@case}
\makeatother

\makeatletter
\def\thestep{\@arabic\c@step}
\makeatother

\newcount\hh
\newcount\mm
\mm=\time
\hh=\time
\divide\hh by 60
\divide\mm by 60
\multiply\mm by 60
\mm=-\mm
\advance\mm by \time
\def\hhmm{\number\hh:\ifnum\mm<10{}0\fi\number\mm}

\let\oldmarginpar\marginpar
\renewcommand\marginpar[1]{\-\oldmarginpar[\raggedleft\footnotesize #1]%
{\raggedright\footnotesize #1}}

\newcommand\NN{\mathbb{N}}

\newcommand\RR{\mathbb{R}}

\newcommand\cB{{\mathcal{B}}}

\newcommand\cD{{\mathcal{D}}}

\newcommand\cF{{\mathcal{F}}}

\newcommand\cL{{\mathcal{L}}}

\newcommand\cN{{\mathcal{N}}}

\newcommand\cR{{\mathcal{R}}}

\newcommand\cW{{\mathcal{W}}}

\newcommand\fa{{\mathfrak{a}}}

\newcommand\fb{{\mathfrak{b}}}

\newcommand\fc{{\mathfrak{c}}}

\newcommand\fe{{\mathfrak{e}}}

\newcommand\ff{{\mathfrak{f}}}

\newcommand\fl{{\mathfrak{l}}}

\newcommand{\pa}{\partial}

\newcommand\eps{\varepsilon}

\newcommand\restrictedto{\upharpoonright}

\numberwithin{equation}{section}

\newcommand{\Kim}{\operatorname{Kim}}
\newcommand{\WF}{\operatorname{WF}}
\newcommand{\Id}{\operatorname{Id}}

\usepackage{amssymb}
\usepackage{hyperref}
\usepackage{mathrsfs}
\usepackage{slashed}
\usepackage[usenames]{color}
\usepackage{url}
\usepackage{verbatim}
\usepackage{tikz}
\usepackage{times}

\hypersetup{pdftitle={Boundary estimates for degenerate equations}}

\hypersetup{pdfauthor={Charles L. Epstein and Camelia A. Pop}}

\begin{document}

\title[Boundary estimates for degenerate equations]
{Boundary estimates for a degenerate parabolic equation with partial Dirichlet boundary conditions}

\author[C.L.~Epstein]{Charles L.~Epstein}
\address[CLE]{Department of Mathematics, University of Pennsylvania, 209 S. 33rd Street, Philadelphia, PA, 19104-6395}
\email{cle@math.upenn.edu}
\author[C.A.~Pop]{Camelia A.~Pop}
\address[CP]{School of Mathematics, University of Minnesota, 206 Church St. SE, Minneapolis, MN 55455}
\email{capop@umn.edu}

\date{\today{ }\hhmm}

\begin{abstract}
We study the boundary regularity properties and derive pointwise a priori supremum estimates of weak solutions and their derivatives in terms of suitable weighted $L^2$-norms for a class of degenerate parabolic equations that satisfy homogeneous Dirichlet boundary conditions on certain portions of the boundary. Such equations arise in population genetics in the study of models for the evolution of gene frequencies. Among the applications of our results is the description of the structure of the transition probabilities and of the hitting distributions of the underlying gene frequencies process.
\end{abstract}

%

\dedicatory{This paper is dedicated to the memory of Gennadi M.~Henkin (1942-2016).}
\thanks{CLE research partially supported by the NSF under grant DMS-1507396,
  and by the ARO under grant W911NF-12-1-0552} 
\thanks{CP gratefully acknowledges the support and hospitality provided by the
  IMA during the academic year 2015-2016.}

\maketitle

\tableofcontents

\section{Introduction}
\label{sec:Introduction}

We analyze the boundary behavior of solutions to a class of degenerate
parabolic equations defined on compact manifolds with corners
\cite{Melrose_1991}. The type of weak solutions that we consider satisfy
homogeneous Dirichlet boundary conditions on suitable portions of the boundary
of the manifold. Our results include the proof of a priori pointwise supremum
estimates of weak solutions in terms of the weighted $L^2$-norm of the solution
and of boundary Harnack principles. In \cite {Epstein_Pop_2015} we apply the
results of this article to give a detailed description of the structure of the
fundamental solution for the heat equation and of the caloric measure. The
family of operators we study are called \emph{generalized Kimura operators},
and they were introduced in the work of C.~L.~Epstein and R.~Mazzeo
\cite{Epstein_Mazzeo_annmathstudies}. A local description of a generalized
Kimura operators defined on a compact manifold $P$ with corners can be given in
an adapted system of coordinates by:
\begin{equation}
\label{eq:Operator}
\begin{aligned}
Lu &= \sum_{i=1}^n\left(x_i\bar a_{ii}(z)u_{x_ix_i} +b_i(z) u_{x_i}\right) +\sum_{i,j=1}^n x_ix_ja_{ij}(z)u_{x_ix_j}+\sum_{l,k=1}^md_{lk}(z)u_{y_ly_k}\\
&\quad+\sum_{i=1}^n\sum_{l=1}^mx_ic_{il}(z)u_{x_iy_l} + \sum_{l=1}^me_l(z)u_{y_l} +c(z)u,
\end{aligned}
\end{equation}
where we identify a boundary point in $\partial P$ with the origin in $\bar S_{n,m}:=\RR^n_+\times\RR^m$ ($\RR_+:=(0,\infty), n,m\in\NN$), and we denote $z=(x,y)\in S_{n,m}$. 

\subsection{Boundary behavior of local weak solutions}
\label{sec:Local_sup_est}

In this section we state the main results about the boundary regularity of solutions to the parabolic Dirichlet problem for the operator $L$ defined in \eqref{eq:Operator}. These results contain the proof the the pointwise boundary estimates of solutions (and their 
higher-order derivatives) in Theorem \ref{thm:Boundary_reg} and of the boundary Harnack principles in Theorems \ref{thm:Carleson_Hopf_Oleinik} and \ref{thm:Holder_cont}.

Because we study the local regularity of solutions to equations defined by the operator 
$L$, it is sufficient to impose conditions on the coefficients of $L$ only on a neighborhood of the origin in $\bar S_{n,m}$, say $\bar\cB_2$, where for all $r>0$ and $z^0\in\bar S_{n,m}$ we use the notation
\begin{equation*}
\cB_r(z^0) := \{z \in S_{n,m}: |x_i-x^0_i| < r, \hbox{ for all } 1\leq i \leq n, |y_l-y^0_l|<r,\hbox{ for all } 1\leq l\leq m\}.
\end{equation*}
When $z^0$ is the origin in $\bar S_{n,m}$, we denote for brevity $\cB_r(z^0)$ by $\cB_r$. We introduce:

\begin{assump}
\label{assump:Coeff}
The coefficients of the operator $L$ defined in \eqref{eq:Operator} satisfy:
\begin{enumerate}
\item[1.] The functions $\bar a_{ii}(z)$, $a_{ij}(z)$, $b_i(z)$, $c_{il}(z)$, $d_{lk}(z)$, $e_l(z)$, and $c(z)$ are smooth and bounded functions on $\bar \cB_2$, for all $1\leq i, j\leq n$ and $1\leq l, k \leq m$.
\item[2.] The drift coefficients $b_i(z)$ satisfy the \emph{cleanness
    condition}: There is a positive constant, $\beta_0$, such that for all
  $1\leq i\leq n$ we have that either
\begin{equation}
\label{eq:Cleanness}
\begin{aligned}
b_i(z) &= 0,\quad\forall\, z\in \partial \cB_2\cap\{x_i = 0\},\\
\hbox{or }\quad
b_i(z) &\geq \beta_0 >0,\quad\forall\, z\in \partial \cB_2\cap\{x_i = 0\}.
\end{aligned}
\end{equation}
\item[3.] The \emph{strict ellipticity} condition holds: there is a positive constant, $\lambda$, such that for all $z \in \bar\cB_2$, $\xi\in\RR^n$, and $\eta\in\RR^m$, we have that
\begin{equation}
\label{eq:Uniform_ellipticity}
\begin{aligned}
\sum_{i=1}^n \bar a_{ii}(z)\xi_i^2
+ \sum_{i,j=1}^n a_{ij}(z)\xi_i\xi_j
+\sum_{i=1}^n \sum_{l=1}^m c_{il}(z)\xi_i\eta_l 
+\sum_{l,k=1}^m d_{lk}(z)\eta_k\eta_l
\geq \lambda\left(|\xi|^2+|\eta|^2\right).
\end{aligned}
\end{equation}
\end{enumerate}
\end{assump}

From \eqref{eq:Uniform_ellipticity}, we see that the operator $L$ is not
strictly elliptic up to the boundary $\partial S_{n,m}$, because the
coefficients of the second-order derivatives $\partial^2_{x_i}$ are linearly
proportional to the distance to the boundary, and so they converge to $0$ as we
approach the boundary component $\{x_i=0\}$. We also notice that the terms
$x_i\partial^2_{x_i}$ and $b_i(z)\partial_{x_i}$ scale in the same way, and so,
at the boundary $\{x_i=0\},$
the first-order derivative $b_i(z)\partial_{x_i}$ is not of lower order, as
in the case of strictly elliptic operators. The coefficients $\{b_i:1\leq i\leq
n\}$ play a fundamental role in the analysis and this is hinted also in our
description of the behavior of the Wright-Fisher process on boundary components
of the simplex, $\{x_i=0\}\cap\bar\Sigma_n$ with $b_i=0$, where the process is
absorbed instead of being reflected. On such absorbent boundary components, we
impose homogeneous Dirichlet boundary conditions.

Applying \cite[Proposition 2.2.3]{Epstein_Mazzeo_annmathstudies}, we can make a change of the coordinate system so that we can assume without loss of generality that the operator $L$ defined in \eqref{eq:Operator} has the property that
\begin{equation}
\label{eq:Operator_adapted_system}
\bar a_{ii}(z) = 1,\quad\forall\, z\in \bar \cB_2,\quad\forall\, 1\leq i\leq n.
\end{equation}
The coefficients $\{b_i:1\leq i\leq n\}$ play a fundamental role in the
analysis of generalized Kimura operators.  Following \cite[Definition
2.1]{Epstein_Mazzeo_2016}, in coordinates
satisfying~\eqref{eq:Operator_adapted_system}, the coefficient $b_i$ restricted
to the boundary component $\{x_i=0\}\cap \partial\cB_2$ is called a weight of
the operator $L$. Notice that by condition \eqref{eq:Cleanness} in Assumption
\ref{assump:Coeff}, we can assume without loss of generality that there is
$n_0\in\NN$ such that
\begin{equation}
\label{eq:n_0}
b_i\restrictedto_{\{x_i=0\}\cap \partial\cB_2} =0,\quad\forall\, 1\leq i\leq n_0,
\quad\hbox{ and }\quad
b_i\restrictedto_{\{x_i=0\}\cap \partial\cB_2} \geq\beta_0>0,\quad\forall\, n_0+1\leq i\leq n.
\end{equation}
The notion of weak solutions we consider, and whose technical definition we defer to \S \ref{sec:Weak_sol}, satisfies homogeneous Dirichlet boundary conditions corresponding to the portions of the boundary with zero weights (i.e., $1\leq i\leq n_0$). However, we impose no boundary conditions along the portions of the boundary with positive weights (i.e., $n_0+1\leq i\leq n$). 

\subsubsection{Pointwise supremum estimates}
\label{sec:Pointwise_sup_est}

To state the a priori pointwise supremum estimates satisfied by weak solutions, we first need to introduce additional notation.
We use the coefficients $\{b_i(z):1\leq i\leq n\}$ to define the measure:
\begin{equation}
\label{eq:Weight}
d\mu(z) = \prod_{i=1}^{n_0} x_i^{-1}\, dx_i\prod_{j=n_0+1}^n x_j^{b_j(z)-1}\, dx_j\prod_{l=1}^m dy_l,
\end{equation}
where we recall the definition of $n_0$ in \eqref{eq:n_0}.
Given a Borel measurable set, $\Omega\subseteq S_{n,m}$, a measurable function $u:\Omega\rightarrow\RR$ is said to belong to the space of functions $L^2(\Omega;d\mu)$ if the norm
\begin{equation}
\label{eq:L_2_norm_S_n_m}
\|u\|^2_{L^2(\Omega;d\mu)} := \int_{\Omega} |u(z)|^2 \, d\mu(z) <\infty.
\end{equation}
Let $\fa\in\NN^n$ and $\fb\in\NN^m$. We denote by $\fe_i\in\NN^n$ the unit vector in $\RR^n$ with all coordinates 0, except for the $i$-th coordinate, which is equal to 1. We denote by $\ff_l\in\NN^m$ the unit vector in $\RR^m$ with all coordinates 0, except for the $l$-th coordinate, which is equal to 1. For all $\fa=(\fa_1,\ldots,\fa_n)\in\NN^n$, $\fb=(\fb_1,\ldots,\fb_m)\in\NN^m$, and $\fc\in\NN$ we denote
$$
D^{\fa}_xD^{\fb}_yD^{\fc}_t := 
\frac{\partial^{|\fa|}}{\partial x_1^{\fa_1}\ldots\partial x_n^{\fa_n}}
\frac{\partial^{|\fb|}}{\partial y_1^{\fb_1}\ldots\partial y_m^{\fb_m}}
\frac{\partial^{\fc}}{\partial t^{\fc}},
$$
where $|\fa|:=|\fa_1|+\ldots+|\fa_n|$ and $\fb=|\fb_1|+\ldots+|\fb_m|$. 

We can now state the first main result in which we establish \emph{pointwise boundary supremum estimates} of local weak solutions to the parabolic Dirichlet problem defined by the operator $L$. 

\begin{thm}[Boundary regularity]
\label{thm:Boundary_reg}
Assume that the operator $L$ in \eqref{eq:Operator} satisfies \eqref{eq:Operator_adapted_system} and Assumption \ref{assump:Coeff}. Let $T>0$, $R\in (0,1)$, and let $u\in L^2((0,T);L^2(\cB_R;d\mu))$ be a local weak solution to equation 
\begin{equation}
\label{eq:Problem_local}
u_t-Lu=0  \hbox{ on } (0, T)\times \cB_R. 
\end{equation}
Then we have that
\begin{equation}
\label{eq:Boundary_reg}
u \in C^{\infty}((0,T)\times\bar \cB_r),\quad\forall\, r\in (0,R),
\end{equation}
and for all $\fa\in\NN^n$, $\fb\in\NN^m$, and $\fc\in\NN$, there is a positive constant,
$C=C(\fa,\fb,\fc,L,R,t,T)$, such that 
\begin{equation}
\label{eq:Boundary_est}
\|D^{\fa}_xD^{\fb}_yD^{\fc}_t u\|_{C([t,T]\times\bar\cB_{R/4})} \leq C \|u\|_{L^2((0,T);L^2(\cB_R;d\mu))}.
\end{equation}
Moreover, for all $s\in [t,T]$ and for all $z\in \bar\cB_{R/4}$, the more exact pointwise estimates hold:
\begin{align}
\label{eq:Sup_est_sol_1_boundary}
|D^{\fb}_y u(s,z)| &\leq C \|u\|_{L^2((0,T);L^2(\cB_R;d\mu))} \prod_{i=1}^{n_0} x_i,
\end{align}
and for all $1\leq k\leq n_0$, we have that
\begin{align}
\label{eq:Sup_est_sol_2_boundary}
|D^{\fe_k}_xD^{\fb}_y u(s,z)| 
&\leq C \|u\|_{L^2((0,T);L^2(\cB_R;d\mu))} \prod_{\stackrel{i=1}{i \neq k}}^{n_0} x_i,
\end{align}
and for all $n_0+1\leq l\leq n$, we have that
\begin{align}
\label{eq:Sup_est_sol_3_boundary}
|D^{\fe_l}_xD^{\fb}_yu(s,z)| 
&\leq C \|u\|_{L^2((0,T);L^2(\cB_R;d\mu))} \prod_{i=1}^{n_0} x_i.
\end{align}
\end{thm}

The difficulty in establishing the boundary regularity properties stated in
Theorem \ref{thm:Boundary_reg} arises from the fact that the operator is not
strictly elliptic up to the boundary $\partial S_{n,m}.$ Moreover, the boundary
of the domain $S_{n,m}$ is non-smooth, and the weight function defined in
\eqref{eq:Weight} is singular; when $n_0>0$
$$
\int_{\cB_r} d\mu = \infty,\quad\forall\, r>0.
$$
Our approach to establish the boundary regularity is based on first proving higher-order regularity in weighted Sobolev spaces, which we then combine with a conjugation property of generalized Kimura operators, know in probability under the name of Doob's $h$-transform \cite{Doob_1957} to obtain the pointwise estimates in Theorem \ref{thm:Boundary_reg}. This property was previously used in the study of particular cases of Kimura operators in articles such as \cite{Shimakura_1981, Epstein_Wilkening_2015}.

We can compare our Theorem \ref{thm:Boundary_reg} with the supremum estimates
of solutions derived by C.~L.~Epstein and R.~Mazzeo in
\cite{Epstein_Mazzeo_2016}, in the case when the weights $\{b_i:1\leq i\leq
n\}$ of the operator $L$ are all positive. Our estimates extend the supremum
estimates in \cite{Epstein_Mazzeo_2016} because: (a) we allow the weights of
the operator to be $0$ also; (b) we prove estimates of the solution and also of
its higher order derivatives; (c) we give a pointwise description of the
boundary behavior of solutions; (d) we derive a boundary Harnack principle
which implies that the pointwise estimates in Theorem \ref{thm:Boundary_reg}
are optimal. Moreover, our method of proof of the supremum estimates is
completely different from that in \cite{Epstein_Mazzeo_2016}, which is based on
employing the Moser iteration method. When the operator $L$ has zero weights,
the measure \eqref{eq:Weight} is non-finite and non-doubling, and so it is not
possible to prove the estimates stated in Theorem \ref{thm:Boundary_reg} using
ideas based on Moser or De Giorgi iterations.

\subsubsection{Boundary Harnack principle}
\label{sec:Boundary_Harnack_princ}

In this section we state our main results concerning the \emph{boundary Harnack principle} satisfied by nonnegative local weak solutions. These consists in the proof of a Carleson-type estimate \eqref{eq:Carleson} and of a local boundary comparison principle \eqref{eq:Quotient_bounds_sup}. Estimates \eqref{eq:Carleson} and \eqref{eq:Quotient_bounds_sup} are an extension to the class of degenerate Kimura operators of the corresponding estimates satisfied by nonnegative solutions to parabolic problems defined by strictly elliptic operators established in \cite[Theorem 3.1]{Salsa_1981}, \cite[Theorem 1.1]{Fabes_Garofalo_Salsa_1986}, \cite[Theorem 3.3]{Fabes_Safonov_Yuan_1999}, \cite[Theorem 2.3]{Garofalo_1984} and \cite[Theorem 1.6]{Fabes_Garofalo_Salsa_1986}, \cite[Theorem 5]{Fabes_Safonov_1997}, \cite[Theorem 4.3]{Fabes_Safonov_Yuan_1999}, \cite[Theorems 3.1 and 3.2]{Garofalo_1984}, respectively. Moreover, we derive a Hopf-Oleinik-type estimate \eqref{eq:Hopf_Oleinik}, which is an extension to the class of the degenerate Kimura operators of the parabolic Hopf-Oleinik boundary principle for strictly elliptic operators, \cite{Friedman_1958, Viborni_1957}. Our method of the proof is based on the conjugation property of generalized Kimura operators described in \S\ref{sec:Supremum_estimates}, and circumvents any use of estimates of the fundamental solution, of the relationship between the Green's function and the caloric measure, or of Landis-type growth estimates.

To state the results we need to introduce the following notation. For a point $(t,z)\in(0,T)\times\partial S_{n,m}$ and 
$r<\sqrt{t}/2$, we denote
\begin{align*}
Q^+_r(t,z)&:=(t+r^2,t+2r^2)\times B_r(z),\\
Q_r(t,z)&:=(t-r^2,t)\times B_r(z),\\
Q^-_r(t,z)&:=(t-3r^2,t-2r^2)\times B_r(z),
\end{align*}
where $B_r(z):=\{z'\in S_{n,m}:\,\rho(z,z')<r\}$ and $\rho(z,z')$ denotes a distance function equivalent with the intrinsic Riemannian metric defined by the principal symbol of the operator $L$, and is given by
\begin{equation}
\label{eq:Distance_function}
\rho(z,z')=\left(\sum_{i=1}^n\left|\sqrt{x_i}-\sqrt{x_i'}\right|^2 + \|y-y'\|^2\right)^{1/2},
\end{equation}
for all $z=(x,y)$ and $z'=(x',y')$ in $\bar S_{n,m}$. Using the inequality $\sqrt{a+b}-\sqrt{b}\leq \sqrt{a}$, for all $a,b\geq 0$, it follows from definition \eqref{eq:Distance_function} that the point
\begin{equation}
\label{eq:A_r}
A_r(z):=\left(x_1+\frac{r^2}{4n},\ldots,x_n+\frac{r^2}{4n},y_1+\frac{r}{2\sqrt{m}},\ldots,y_m+\frac{r}{2\sqrt{m}}\right)
\end{equation}
belongs to $B_r(z)$, for all $z\in\partial S_{n,m}$. Let
\begin{equation}
\label{eq:Weight_density_tangent}
w^T(z) := \prod_{i=1}^{n_0} x_i^{-1}.
\end{equation}

\begin{thm}
\label{thm:Carleson_Hopf_Oleinik}
Assume that the operator $L$ in \eqref{eq:Operator} satisfies \eqref{eq:Operator_adapted_system} and Assumption \ref{assump:Coeff}. 
There is a positive constant, $H=H(\fb,L,n,m)$, such that for all $z\in(\partial\cB_1\cap\partial S_{n,m})$, $r\in (0,1)$, and $t>4r^2$, if $u$ is a nonnegative local weak solution to equation: 
$$
u_t-Lu=0\hbox{ on } (t-4r^2,t+4r^2)\times B_{2r}(z),
$$
then the following hold:
\begin{enumerate}
\item[(i)] (Carleson-type estimate)
\begin{equation}
\label{eq:Carleson}
\sup_{Q_r(t,z)} w^T u \leq H w^T(A_r(z)) u(t+r^2, A_r(z)).
\end{equation}
\item[(ii)] (Hopf-Oleinik-type estimate)
\begin{equation}
\label{eq:Hopf_Oleinik}
\inf_{Q_r(t,z)} w^T u \geq H w^T(A_r(z)) u(t-2r^2, A_r(z)).
\end{equation}
\item[(iii)] (Quotient bounds)
\begin{align}
\label{eq:Quotient_bounds_sup}
\sup_{Q_r(t,z)} \frac{u_1}{u_2} 
&\leq H \frac{u_1(t+r^2, A_r(z))}{u_2(t-2r^2, A_r(z))},\\
\label{eq:Quotient_bounds_inf}
\inf_{Q_r(t,z)} \frac{u_1}{u_2} 
&\geq H^{-1} \frac{u_1(t-2r^2, A_r(z))}{u_2(t+r^2, A_r(z))}. 
\end{align}
\end{enumerate}
\end{thm}

\begin{rmk}
\label{rmk:Hopf_Oleinik}
The Hopf-Oleinik-type boundary estimate \eqref{eq:Hopf_Oleinik} and definition \eqref{eq:Weight_density_tangent} of the weight $w^T(z)$ show that, when $u$ is a positive solution on $(t-4r^2,t+4r^2)\times B_r(z)$, then the pointwise supremum estimate \eqref{eq:Sup_est_sol_1_boundary} (applied with $\fb=0$) is optimal.
\end{rmk}

For a closed set $\Omega\subseteq [0,\infty)\times\bar S_{n,m}$, the anisotropic H\"older space $C^{\alpha}_{WF}(\Omega)$ consists of functions $u$ such that
\begin{equation}
\label{eq:Holder_space}
\|u\|_{C^{\alpha}_{WF}(\Omega)} := \sup_{(s,z)\in\Omega} |u(s,z)| 
+ \sup_{\stackrel{(s_1,z_1)\neq(s_2,z_2)}{(s_i,z_i)\in\Omega, i=1,2}} 
\frac{|u(s_1,z_1)-u(s_2,z_2)|}{\left(\sqrt{|s_1-s_2|}+\rho(z_1,z_2)\right)^{\alpha}}.
\end{equation}
We can now state the analogue of the boundary comparison principle for nonnegative solutions defined by strictly elliptic operators \cite[Theorem 7]{Fabes_Safonov_1997}, \cite[Theorems 4.5 and 4.6 ]{Fabes_Safonov_Yuan_1999} for the class of degenerate Kimura operators.

\begin{thm}[Boundary comparison principle]
\label{thm:Holder_cont}
Assume that the operator $L$ in \eqref{eq:Operator} satisfies \eqref{eq:Operator_adapted_system} and Assumption \ref{assump:Coeff}. 
There is a positive constant, $\alpha=\alpha(\fb,L,n,m)\in (0,1)$, such that for all $z\in(\partial\cB_1\cap\partial S_{n,m})$, $r\in (0,1)$, and $t>4r^2$, if $u_1$ is a nonnegative local weak solution, and $u_2$ is a \emph{positive} local weak solution  to equation: 
$$
u_t-Lu=0\hbox{ on } (t-4r^2,t+4r^2)\times B_{2r}(z),
$$
then we have that
\begin{equation}
\label{eq:Holder_cont}
\frac{u_1}{u_2} \in C^{\alpha}_{\hbox{\tiny{WF}}}(\bar Q_r(t,z)).
\end{equation}
\end{thm}

\subsection{Boundary regularity of global solutions}
\label{sec:Global_regularity}
Before proceeding to the main result of this section about the regularity up to the boundary and the Harnack-type inequalities satisfied by \emph{global} solutions to the parabolic Dirichlet problem defined by a degenerate second-order operator, $\cL$, we review the definitions used in the statement of the result, which were introduced in \cite[\S 2]{Epstein_Mazzeo_2016} and \cite[\S 2.1]{Epstein_Mazzeo_annmathstudies}.

Let $P$ be a compact manifold with corners of dimension $N$. Given a point $p\in P$, we can find nonnegative integers, $n, m\in \NN$, such that $n+m=N$, and we can choose a local system of coordinates, $\psi:U\rightarrow V$, such that $U$ is a relatively open neighborhood of $p$ in $P$, $V$ is a relatively open neighborhood of the origin in $\bar S_{n,m}$, and $\psi$ is a homeomorphism with the property that $\psi(p)=0$. In this section we consider second-order differential operators $\cL$ defined on compact manifolds with corners $P$ with the property that, when written in a local system of coordinates on $P$, the operator $\cL$ takes the form of the operator $L$ in \eqref{eq:Operator}. In addition, applying \cite[Proposition 2.2.3]{Epstein_Mazzeo_annmathstudies}, we can find a local system of coordinates such that the operator $L$ satisfies conditions \eqref{eq:Operator_adapted_system}. Following \cite[\S 2.1]{Epstein_Mazzeo_annmathstudies}, we call this a normal form of the operator $L$ and the coordinate system is said to be adapted. 

The set of points $p\in P$ that have a local system of coordinates that map a neighborhood of the point $p$ onto a neighborhood of the origin in $\bar S_{1,N-1}$ lie in an open smooth manifold of co-dimension $1$. Such a manifold can be written as a disjoint union of open smooth connected manifolds of co-dimension $1$, which we call the boundary hypersurfaces or faces of $P$. We denote the closure of the connected boundary hypersurfaces of $P$ by $H_1, H_2,\ldots, H_k$. Recall from \cite[\S 2.2]{Epstein_Mazzeo_annmathstudies} that the principal symbol of the operator $\cL$ induces a Riemannian metric on the manifold $P$. Letting $H_i$ be a boundary hypersurface and defining $\rho_i(p)$ to be the Riemannian distance from the point $p\in P$ to $H_i$, we recall from \cite[Proposition 2.1]{Epstein_Mazzeo_2016} that
\begin{equation}
\label{eq:Weight_b_i}
B_i\restrictedto_{H_i} := L\rho_i\restrictedto_{H_i},\quad\forall\, 1\leq i\leq k,
\end{equation}
are coordinate-invariant quantities, and we call them the \emph{weights} of the generalized Kimura operator, \cite[Definition 2.1]{Epstein_Mazzeo_2016}.

Analogously to \S \ref{sec:Local_sup_est}, we choose smooth extensions of the
weights $\{B_i:1\leq i\leq k\}$ of the operator $\cL$ from the boundary
hypersurfaces of $P$ to the interior of $P$. By an abuse of notation, we denote
the smooth extension of the weight corresponding to the boundary hypersurface
$H_i$ by $B_i$, for $1\leq i\leq k$, and we sometimes refer to them also as the
weights of the generalized Kimura operator. Fixing a sufficiently small
$\eta>0,$ each $B_i(p)$ can be taken, for $\rho_i(p)<\eta,$ to be independent
of the distance to the boundary. For each $1\leq i\leq k$, we then smoothly
interpolate the  extended weight to the constant value $1,$ so that the
function $\rho_i(p)^{B_i(p)-1}$ is globally defined and positive on $P\setminus
H_i.$ That this is possible follows from the tubular neighborhood theorem for
manifolds with corners, Lemma 2.1.3 in~\cite{Epstein_Mazzeo_annmathstudies}.
Using these extended weights, $\{B_i:1\leq i\leq k\},$ we define a measure on the
manifold $P$ by setting
\begin{equation}
\label{eq:Weight_P}
d\mu := \prod_{i=1}^k \rho_i(p)^{B_i(p)-1}\, dV,
\end{equation}
where $dV$ is a smooth positive density on $P$. Note that, in adapted
coordinates $(x_1,\dots,x_n;y_1,\dots,y_m)$ near a boundary point of
co-dimension $n,$ this measure takes exactly the form given in~\eqref{eq:Weight}.

A different smooth choice,
$\{B'_i:1\leq i\leq k\}$, of the extension of the weights from the boundary
hypersurfaces, which are again locally independent of the distance to the
boundary of $P,$ and another smooth nondegenerate choice of a density $dV'$ on
$P$, generate a weighted measure $d\mu'$,
$$
d\mu' := \prod_{i=1}^k \rho_i(p)^{B'_i(p)-1}\, dV',
$$
which differs from $d\mu$ by a bounded, smooth, positive factor on $P$. Our
results are independent of these choices.
We say that a measurable function $u:P\rightarrow\RR$ belongs to $L^2(P; d\mu)$ if the norm
\begin{equation}
\label{eq:L_2_norm_P}
\|u\|^2_{L^2(P;d\mu)} := \int_{P} |u(p)|^2 \, d\mu(p) <\infty.
\end{equation}
We can now state

\begin{thm}[Global regularity of solutions]
\label{thm:Global_regularity}
Assume that $\cL$ is a second-order differential operator defined on a compact
manifold with corners $P$ such that when written in a local system of
coordinates it takes the form of the operator $L$ defined in
\eqref{eq:Operator} and it satisfies Assumption \ref{assump:Coeff}. Let $f \in
L^2(P; d\mu)$ and let $u$ be the unique weak solution to the parabolic
Dirichlet problem,
\begin{equation}
\label{eq:Initial_value_problem}
\begin{aligned}
\left\{\begin{array}{rl}
u_t-\cL u=0 & \hbox{ on } (0, \infty)\times P,\\ 
u(0) = f& \hbox{ on } P, 
\end{array} \right.
\end{aligned}
\end{equation}
Then we have that
\begin{equation}
\label{eq:Global_regularity}
u \in C^{\infty}((0,\infty)\times P).
\end{equation}
\end{thm}
\begin{rmk}
As a corollary of this theorem we can show that the resolvent of the graph
closure of $\cL$ with respect to $C^0(P)$ has a compact resolvent.
\end{rmk}

A similar result was obtained in \cite[Theorem 9.1]{Hofrichter_Tran_Jost_2014a}
for the special case of the classical Kimura operator
\begin{equation}\label{eq:WF_operator}
  \cL_{\Kim}=\sum_{1\leq i,j\leq n}x_i(\delta_{ij}-x_j)\pa_{x_i}\pa_{x_j}
\end{equation}
  acting on functions defined on the $n$-simplex. In \cite[Theorem 9.1]{Hofrichter_Tran_Jost_2014a}, the authors prove the smoothness of solutions for positive time to the parabolic equation for the Wright-Fisher operator, but with initial data in $L^2(\Sigma_n)$, as opposed to $L^2(\Sigma_n;d\mu)$. 

  It is interesting to contrast Theorem \ref{thm:Global_regularity} with the
  boundary regularity of solutions to the elliptic Dirichlet problem. In
  \cite[Theorem 4.8]{Epstein_Wilkening_2015}, the authors study the elliptic
  \emph{non-homogeneous} Dirichlet problem for the classical Kimura operator,
  proving that solutions can have mild logarithmic singularities at the
  boundary, which are sums of terms of the form $(x_{i_1}+\ldots+x_{i_k}) \ln
  (x_{i_1}+\ldots+x_{i_k})$, for all $1\leq i_1<i_2<\ldots<i_k\leq n$ and for
  all $1\leq k\leq n$.

An important application of Theorem \ref{thm:Global_regularity} is that the Dirichlet heat kernel associated to Kimura operators that are tangent to all boundary components is smooth. We prove this result in 
\cite[Theorem 1.6]{Epstein_Pop_2015}.

We next state a boundary comparison estimate satisfied by nonnegative global solutions, which is an extension to the class of the degenerate Kimura operators of the corresponding result for strictly elliptic operators in divergence form \cite[Theorem 1.7]{Fabes_Garofalo_Salsa_1986}. 

\begin{thm}[Quotient bounds for global nonnegative solutions]
\label{thm:Quotient_bounds_global}
For $i=1,2$, let $u_i$ be global nonnegative weak solutions to $(\partial_t-\cL)u_i=0$ on $(0,T)\times P$. Assume that $u_2$ is a positive solution on $\hbox{int}(P)$. Then for all $r\in (0,1)$, there is a positive constant, $H=H(\fb,L,n,m,r)$, such that for all $4r^2<t<T-4r^2$ and for all $p\in \partial P$ we have that
\begin{equation}
\label{eq:Quotient_bounds_global}
\sup_{Q_r(t,p)} \frac{u_1}{u_2} \leq H \inf_{Q_r(t,p)} \frac{u_1}{u_2}.
\end{equation}
\end{thm}

\subsection{Applications}
\label{sec:Applications}

Generalized Kimura operators arise in population genetics as a model for the
evolution of gene frequencies as diffusion processes, \cite{Fisher_1930,
  Haldane_1932, Wright_1931, Kimura_1964, hartl1997principles,
  Epstein_Mazzeo_annmathstudies}. To describe the statistical properties of
such processes, we need to gain a good understanding of their transition
probabilities and of the hitting distributions on suitable portions of the
boundary of the support of the process. Such questions are equivalent to the
understanding of the fundamental solution and of the caloric measure associated
to the parabolic problem that we study in this article. We carry out this
analysis in \cite[Theorem 1.10]{Epstein_Pop_2015}, where one of our main
arguments to establish the structure of the transition probabilities of Kimura
diffusions (fundamental solution) relies on the pointwise boundary estimates in
Theorem \ref{thm:Boundary_reg}. In addition, the boundary Harnack principles
established in our present article allow us to prove in \cite[Theorem
5.4]{Epstein_Pop_2015} the doubling property of the hitting distributions of
Kimura diffusions (caloric measure).

Finally, we remark that the class of processes described by generalized Kimura operators appear not only in population genetics, 
but they are also encountered in the study of superprocesses, \cite{Athreya_Barlow_Bass_Perkins_2002, Bass_Perkins_2003}, of Fleming-Viot processes in population dynamics, \cite{Cerrai_Clement_2001, Cerrai_Clement_2003, Cerrai_Clement_2004, Cerrai_Clement_2007}, and 
are closely related to the linearization of the porous medium equation, \cite{DaskalHamilton1998, Koch}, to affine models for interest rates, \cite{Cuchiero_Filipovic_Mayerhofer_Teichmann_2011, DuffiePanSingleton2000}, and to stochastic volatility models in mathematical finance, \cite{Heston1993, Feehan_Pop_regularityweaksoln, Feehan_Pop_mimickingdegen_pde, Feehan_Pop_elliptichestonschauder, Feehan_Pop_higherregularityweaksoln, Feehan_Pop_mimickingdegen_probability,  Feehan_Pop_stochrepdirichlet}.

\section{Weak solutions to the parabolic Dirichlet problem}
\label{sec:Weak_sol}

In this section we review from \cite[\S 3.1 and \S 3.2]{Epstein_Pop_2015} the notion of weak solutions to the parabolic Dirichlet problems introduced in \S\ref{sec:Introduction}, and we recall an energy estimate satisfied by such solutions. In \cite[\S 3.1]{Epstein_Pop_2015}, we prove that we can associate a Dirichlet form to the operator $\cL$ defined on a compact manifold $P$ with corners as introduced in \S \ref{sec:Global_regularity}, which is defined by
$$
Q(u,v) := - (\cL u, v)_{L^2(P;d\mu)},\quad\forall\, u,v \in C^{\infty}_c(\hbox{int}(P)).
$$ 
The bilinear form $Q(u,v)$ can be decomposed as
\begin{equation}
\label{eq:Bilinear_form}
Q(u,v) := Q_{\hbox{\tiny{sym}}}(u,v) + (Vu, v)_{L^2(P; d\mu)} + (cu, v)_{L^2(P; d\mu)},\quad\forall\, u,v \in C^{\infty}_c(\hbox{int}(P)),
\end{equation}
where $c$ is a smooth and bounded zeroth order term on $P$,
$Q_{\hbox{\tiny{sym}}}(u,v)$ is a symmetric bilinear form, and $V$ is a
vector field on the manifold $P$ that is tangent to $\partial P,$
whose coefficients might have mild logarithmic singularities. When written in a
local system of coordinates on a neighborhood $\cB_R$ of the origin in $\bar
S_{n,m}$, the symmetric bilinear form $Q_{\hbox{\tiny{sym}}}(u,v)$ takes the
form
\begin{equation}
\label{eq:Q_sym}
\begin{aligned}
Q_{\hbox{\tiny{sym}}} (u,v) 
&= \int_{\cB_R}\left(\sum_{i=1}^nx_iu_{x_i}v_{x_i} + \sum_{i,j=1}^n \frac{1}{2}x_ix_ja_{ij}(u_{x_i}v_{x_j} + u_{x_j}v_{x_i})\right)\, d\mu\\
&\quad + \int_{\cB_R}\left(\sum_{i=1}^n\sum_{l=1}^m \frac{1}{2}x_ic_{il}(u_{x_i}v_{y_l} + u_{y_l}v_{x_i}) 
+ \sum_{l,k=1}^m \frac{1}{2}d_{lk}(u_{y_l}v_{y_k} + u_{y_k}v_{y_l})\right)\, d\mu,
\end{aligned}
\end{equation}
for all $u,v\in C^1_c(\cB_R)$, and $d\mu$ is the weighted measure in
\eqref{eq:Weight}. 

The boundary is divided into two subsets:
\begin{equation}
\label{eq:Tangent_boundary}
\partial^T P := \cup_{i=1}^k\{\bar H_i : b_i\restrictedto_{H_i} = 0\}\text{ and
}
\partial^{\pitchfork}P=\partial P\setminus \partial^T P.
\end{equation} 
Because the weights corresponding to $\partial^TP$ are constant, the vector
field $V$ satisfies
\begin{equation}
\label{eq:V}
\begin{aligned}
V &= 
\sum_{i=1}^n x_i 
\left(\alpha_i(z) +\sum_{k=n_0+1}^n\left(\sum_{j=1}^n\gamma_{ikj}(z) \partial_{x_j}b_k + \sum_{l=1}^m\nu_{ikl}(z) \partial_{y_l}b_k\right)\ln x_k\right)\partial_{x_i}\\
&\quad+\sum_{l,l'=1}^m\sum_{k=n_0+1}^n \beta_{ll'k}(z)\partial_{y_{l'}} b_k \ln x_k\partial_{y_l} ,
\end{aligned}
\end{equation}
where the functions $\alpha_i,\beta_{ll'k},\gamma_{ikj}, \nu_{ikl}:\bar
\cB_R\rightarrow\RR$ are smooth. 

We define the space of functions $H^1(P;d\mu)$ to be the closure of
$C^{\infty}_c(P\backslash \partial^T P)$ with respect to the norm:
\begin{equation}
\label{eq:H_1_norm}
\|u\|^2_{H^1(P;d\mu)} := Q_{\hbox{\tiny{sym}}} (u, u) + \|u\|^2_{L^2(P;d\mu)}.
\end{equation}
We prove in \cite[Lemma 3.1]{Epstein_Pop_2015} that the bilinear form $Q(u,v)$ is continuous and satisfies the G\r{a}rding inequality, i.e.
\begin{align}
\label{eq:Continuity_Dirichlet_form}
|Q(u,v)| &\leq c_1 \|u\|_{H^1(P;d\mu)} \|v\|_{H^1(P;d\mu)},\\
\label{eq:Coercivity_Dirichlet_form}
Q(u,u) & \geq c_2 \|u\|_{H^1(P;d\mu)}^2 - c_3 \|u\|_{L^2(P;d\mu)}^2,
\end{align}
where $c_1$, $c_2$, and $c_3$ are positive constants depending only on the coefficients of the operator $L$. The preceding properties imply that the hypotheses of \cite[Chapter 3, Section 4, Theorem 4.1 and Remark 4.3]{Lions_Magenes1} are satisfied, and so given $g\in L^2((0,T); L^2(P;d\mu))$ and $f\in L^2(P;d\mu)$, there is a unique weak solution to the parabolic Dirichlet problem,
\begin{equation}
\label{eq:Parabolic_Dirichlet_problem}
\begin{aligned}
\left\{\begin{array}{rl}
u_t-\cL u=g & \hbox{ on } (0, T)\times P,\\ 
u(0) = f& \hbox{ on } P. 
\end{array} \right.
\end{aligned}
\end{equation}
We next recall the definition \emph{global weak solutions} to the parabolic problem \eqref{eq:Parabolic_Dirichlet_problem}: 

\begin{defn}[Global weak solution]
\label{defn:Weak_sol_global}
Let $g\in L^2((0,T); L^2(P;d\mu))$ and $f\in L^2(P;d\mu)$. A function $u$ is a global weak solution to equation \eqref{eq:Parabolic_Dirichlet_problem} if it belongs to the space $\cF((0,T)\times P)$, i.e. 
$$
u \in L^2((0,T); H^1(P;d\mu))
\quad\hbox{ and }\quad
\frac{du}{dt} \in L^2((0,T); H^{-1}(P;d\mu)),
$$
where we denote by $H^{-1}(P;d\mu)$ the dual space of $H^1(P;d\mu)$, and the following hold:
\begin{enumerate}
\item[1.] For all test functions $v\in \cF((0,T)\times P)$, we have that
\begin{equation}
\label{eq:Weak_sol_var_eq}
\int_0^T\left\langle \frac{du(t)}{dt}, v(t)\right\rangle\, dt+\int_0^T Q(u(t),v(t))\, dt= \int_0^T \left(g(t), v(t)\right)_{L^2(P;d\mu)}\, dt,
\end{equation}
where $\left\langle\cdot, \cdot\right\rangle$ denotes the dual pairing of $H^{-1}(P;d\mu)$ and $H^1(P;d\mu)$.
\item[2.] The initial condition is satisfied in the $L^2(P;d\mu)$-sense, that is
\begin{equation}
\label{eq:Weak_sol_initial_cond}
\|u(t)-f\|_{L^2(P;d\mu)}\rightarrow 0,\quad\hbox{as } t\downarrow 0.
\end{equation}
\end{enumerate}
\end{defn}

We have the following remarks about the boundary conditions satisfied by weak solutions to equation 
\eqref{eq:Parabolic_Dirichlet_problem}:

\begin{rmk}[Homogeneous Dirichlet boundary condition along $(0,T)\times\partial^TP$]
\label{rmk:Boundary_cond}
The Dirichlet boundary condition,
\begin{equation}
\label{eq:Boundary_cont}
u = 0 \quad\hbox{ on } (0,T)\times\partial^T P,
\end{equation}
is encoded into the definition of the weighted Sobolev space $H^1(P;d\mu).$ Of
course, functions $v$ in $C^{\infty}_c(P\backslash\partial^T P)$ automatically
satisfy the boundary condition that $v=0$ on $\partial^T P$. From definition
\eqref{eq:Tangent_boundary} of $\partial^T P$, we see that $\partial^T P$
consists of the union of the closed boundary hypersurfaces $H_i$ such that
$B_i\restrictedto_{H_i}=0$. We fix such a boundary hypersurface, $H_i$, and a
relatively open neighborhood, $U\subseteq P$, of a point in
$\hbox{int}(H_i)$. We notice that, when written in local coordinates on $U$,
the operator $\cL$ takes the form of the operator $L$ defined in
\eqref{eq:Operator}, with $n_0=n=1$ and $m=N-1$. Applying Theorem
\ref{thm:Boundary_reg}, we obtain that $u\in C^{\infty}((0,T)\times U)$, and
estimate \eqref{eq:Sup_est_sol_1_boundary} (applied with $\fb=0$) implies that
$u=0$ on $(0,T)\times (U\cap\partial^T P)$. Thus, the homogeneous Dirichlet
boundary condition, $u=0$ on $(0,T)\times \hbox{int}(H_i)$, is satisfied in a
classical sense for all boundary hypersurfaces $H_i$ to which the operator
$\cL$ is tangent. Moreover, by Theorem \ref{thm:Global_regularity} we know that
$u$ is continuous up to the boundary, and so the homogeneous boundary condition
\eqref{eq:Boundary_cont} is satisfied.
\end{rmk}

\begin{rmk}[Boundary condition along $(0,T)\times\partial^{\pitchfork}P$]
\label{rmk:Boundary_cond_transverse}
Along the parabolic portion of the boundary
$(0,T)\times\partial^{\pitchfork}P$, we impose \emph{no boundary condition} in
Definition \ref{defn:Weak_sol_global} of weak solutions. Even so, uniqueness of
weak solutions to the initial-value problem
\eqref{eq:Parabolic_Dirichlet_problem} is not lost because the bilinear form
$Q(u,v)$ satisfies the G\r{a}rding inequality
\eqref{eq:Coercivity_Dirichlet_form}, for all $u,v\in H^1(P;d\mu)$.  Our
motivation to impose no boundary conditions along
$(0,T)\times\partial^{\pitchfork}P$ is because we apply the results in our
article to the characterization of the transition probabilities of generalized
Kimura processes, and an imposition of a boundary condition along
$(0,T)\times\partial^{\pitchfork}P$ would result in altering the natural
boundary behavior of the underlying Kimura process, as is described in
\cite{Epstein_Pop_2015}. In~\cite{Epstein_Mazzeo_annmathstudies} it is shown
that this natural boundary condition can be understood as imposing a regularity
requirement along these boundary components. Similar problems without boundary
conditions on suitable portions of the boundary have been studied in
\cite{DaskalHamilton1998, Koch, Feehan_Pop_mimickingdegen_pde,
  Feehan_Pop_regularityweaksoln, Feehan_Pop_elliptichestonschauder,
  Feehan_maximumprinciple}, among others.
\end{rmk}

In addition \cite[Chapter 4, Section 1, Theorem 1.1]{Lions_Magenes1} implies that the unique weak solution to the parabolic Dirichlet problem \eqref{eq:Parabolic_Dirichlet_problem} satisfies the energy estimate:
\begin{equation}
\label{eq:Energy_estimate}
\begin{aligned}
&\sup_{t\in [0,T]}\|u(t)\|_{L^2(P;d\mu)} + \|u\|_{L^2((0,T);H^1(P;d\mu))} +\left\|\frac{du}{dt}\right\|_{L^2((0,T);H^{-1}(P;d\mu))}\\
&\quad\leq C\left(\|f\|_{L^2(P;d\mu)} + \|g\|_{L^2((0,T);L^2(P;d\mu))}\right),
\end{aligned}
\end{equation}
where $C=C(T, c_1, c_2, c_3)$ is a positive constant and $c_1, c_2, c_3$ are the constants appearing in \eqref{eq:Continuity_Dirichlet_form} and \eqref{eq:Coercivity_Dirichlet_form}. The remark following \cite[Chapter 3, Section 4, Theorem 4.1]{Lions_Magenes1} gives us that $u \in C([0,T];L^2(P;d\mu))$, i.e. 
$$
\|u(t)-u(s)\|_{L^2(P;d\mu)}\rightarrow 0,\quad\hbox{ as } t\rightarrow s,\quad t,s\in[0,T].
$$
In our article, we also use the notion of \emph{local} weak solution, which we define next. For all open sets $\Omega\subset \cB_2\subset S_{n,m}$, we denote
\begin{align*}
\partial^T\Omega:=\cap_{i=1}^{n_0}\{x_i=0\}\cap\partial\Omega,
\quad
\partial^{\pitchfork}\Omega:=\cap_{j=n_0+1}^n\{x_j=0\}\cap\partial\Omega,
\quad
\underline{\Omega} :=\Omega\cup\partial^{\pitchfork}\Omega,
\end{align*}  
and we let $H^1(\Omega;d\mu)$ be closure of $C^{\infty}_c(\underline\Omega)$ with respect to the norm
$$
\|u\|_{H^1(\Omega;d\mu)}:=\int_{\Omega}\left(\sum_{i=1}^n x_i|u_{x_i}|^2 + \sum_{l=1}^m|u_{y_l}|^2 + |u|^2\right)\, d\mu.
$$
As usual $H^{-1}(\Omega;d\mu)$ denotes the dual space of $H^1(\Omega;d\mu)$. Let $Q(u,v)$ be the Dirichlet form defined by \eqref{eq:Bilinear_form} with $P$ replaced by $\Omega$, and let $Q_{\hbox{\tiny{sym}}}(u,v)$ be the Dirichlet form defined by \eqref{eq:Q_sym} with $\cB_R$ replaced by $\Omega$. We can now introduce

\begin{defn}[Local weak solutions]
\label{defn:Weak_sol_local}
(a) Let $T>0$, $g\in L^2((0,T); L^2(\Omega;d\mu))$, and $f\in L^2(\Omega;d\mu)$. A function $u$ is a local weak solution to initial-value problem,
\begin{equation}
\label{eq:Parabolic_Dirichlet_problem_local}
\begin{aligned}
\left\{\begin{array}{rl}
u_t-L u=g & \hbox{ on } (0,T)\times \Omega,\\ 
u(0) = f& \hbox{ on } \Omega, 
\end{array} \right.
\end{aligned}
\end{equation}
if the following hold:
\begin{enumerate}
\item[1.] For all $\varphi\in C^1_c([0,T]\times\underline{\Omega})$, we have that
$$
u\varphi \in L^2((0,T); H^1(\Omega;d\mu))
\quad\hbox{ and }\quad
\frac{d(u\varphi)}{dt} \in L^2((0,T); H^{-1}(\Omega;d\mu)).
$$
\item[2.] Equality \eqref{eq:Weak_sol_var_eq} holds for all test functions $v$ satisfying $v\in L^2((0,T);H^1(\Omega;d\mu))$, $dv/dt\in L^2((0,T); H^{-1}(\Omega;d\mu))$, and $\hbox{supp}(v)\subset [0,T]\times\underline\Omega$.
\item[3.] Property \eqref{eq:Weak_sol_initial_cond} holds with $u$ and $f$ replaced by $u\varphi$ and $f\varphi$, respectively, for all $\varphi\in C^1_c(\underline\Omega)$. 
\end{enumerate}
(b) Let $0<t_1,t_2$ and $g\in L^2((t_1,t_2); L^2(\Omega;d\mu))$. A function $u$ is a local weak solution to equation,
\begin{equation}
\label{eq:Parabolic_Dirichlet_problem_local_int}
u_t-L u=g  \hbox{ on } (t_1,t_2)\times \Omega,\\ 
\end{equation}
if the following hold:
\begin{enumerate}
\item[1.] For all $\varphi\in C^1_c((t_1,t_2]\times\underline{\Omega})$, we have that
$$
u\varphi \in L^2((t_1,t_2); H^1(\Omega;d\mu))
\quad\hbox{ and }\quad
\frac{d(u\varphi)}{dt} \in L^2((t_1,t_2); H^{-1}(\Omega;d\mu)).
$$
\item[2.] Equality \eqref{eq:Weak_sol_var_eq} holds for all test functions $v$ satisfying $v\in L^2((t_1,t_2);H^1(\Omega;d\mu))$, $dv/dt\in L^2((t_1,t_2); H^{-1}(\Omega;d\mu))$, and $\hbox{supp}(v)\subset (t_1,t_2]\times\underline\Omega$.
\end{enumerate}
\end{defn}

\section{Weighted Sobolev estimates}
\label{sec:Weighted_Sobolev_estimates}

In this section we consider the operator $L$ defined in \eqref{eq:Operator} that acts on functions defined on $S_{n,m}$. Our main result is estimate \eqref{eq:Regularity_Sobolev} in which we prove higher-order local regularity in weighted Sobolev spaces of solutions to the initial-value problem \eqref{eq:Initial_value_problem_local}. Estimate \eqref{eq:Regularity_Sobolev} plays a central role in \S \ref{sec:Supremum_estimates} in the proof of the supremum estimates of the derivatives of solutions to the initial-value problem \eqref{eq:Initial_value_problem_local} in terms of the weighted $L^2$-norm of the initial data; see Theorem \ref{thm:Sup_est}.

We remark that estimate \eqref{eq:Regularity_Sobolev} cannot be obtained by standard methods available in the literature, such as finite-difference arguments, because the boundary of the domain $S_{n,m}$ is non-smooth.
To overcome this difficulty, we prove in Lemma \ref{lem:Approx_smooth} an approximation property of weak solutions with smooth functions. The latter is a property that is not in generally true for strictly elliptic operators defined on non-smooth domains. This observation allows us to take derivatives in the equations satisfied by the approximating sequence of smooth solutions. And so our method of the proof consists of first proving the estimates under the stronger assumption that the local solution is smooth and then using the approximation procedure to derive the a priori local Sobolev estimate \eqref{eq:Regularity_Sobolev} in its full generality. 

For $\fa=(\fa_1,\ldots,\fa_n) \in \NN^n$, we introduce the weighted measure,
\begin{equation}
\label{eq:Weight_for_higher_order_Sobolev_spaces}
d\mu_{\fa}(z) := \prod_{i=1}^n x_i^{b_i(z)+\fa_i-1}\, dx_i\prod_{l=1}^mdy_l,\quad\forall\, z=(x,y)\in S_{n,m},
\end{equation}
which is suitable to measure the regularity of the derivatives
$D^{\fa}_xD^{\fb}_yu$ of the weak solutions $u$ to the initial-value problem
\eqref{eq:Initial_value_problem}, for all $\fb\in\NN^m$.   We can now state:

\begin{thm}[Regularity in weighted Sobolev spaces]
\label{thm:Regularity_Sobolev}
Assume that the operator $L$ in \eqref{eq:Operator} satisfies \eqref{eq:Operator_adapted_system} and Assumption \ref{assump:Coeff}. Let $T>0$, $R\in (0,1)$, $f\in L^2(\cB_R; d\mu)$, and let $u$ be a local weak solution to the initial-value problem,
\begin{equation}
\label{eq:Initial_value_problem_local}
\begin{aligned}
\left\{\begin{array}{rl}
u_t-Lu=0 & \hbox{ on } (0, T)\times \cB_R,\\ 
u(0) = f& \hbox{ on } \cB_R.
\end{array} \right.
\end{aligned}
\end{equation}
Then for all $(\fa,\fb)\in \NN^n\times\NN^m$, $0<r<R$, and $0<t<T$, we have that
$$
D^{\fa,\fb}u \in L^{\infty}([t,T], L^2(\cB_r;d\mu_{\fa})),
$$
and there is a positive constant, $C=C(\fa,\fb,L,r,R,t,T)$, such that
\begin{equation}
\label{eq:Regularity_Sobolev}
\sup_{s\in [t,T]} \|D^{\fa}_xD^{\fb}_yu(s)\|^2_{L^2(\cB_r;d\mu_{\fa})} 
\leq C \left(\|f\|^2_{L^2(\cB_R;d\mu)} + \|u\|^2_{L^2((0,T);L^2(\cB_R;d\mu))}\right).
\end{equation}
\end{thm}

The proof of Theorem \ref{thm:Regularity_Sobolev} is based on a series of auxiliary results, which we prove in the sequel. Because of the local nature of the problem, we will use a smooth cutoff function, $\varphi:\bar S_{n,m}\rightarrow [0,1]$, such that 
\begin{equation}
\label{eq:Cutoff_Sobolev_est}
\varphi\equiv 1\hbox{ on } \bar\cB_r\quad\hbox{ and }\quad \varphi \equiv 0\hbox{ on } \bar\cB_{r_0}^c,
\end{equation}
where $0<r<r_0<R$. We begin with

\begin{lem}[First-order derivatives estimates]
\label{lem:First_order_derivatives}
Assume that the operator $L$ satisfies \eqref{eq:Operator_adapted_system} and Assumption \ref{assump:Coeff}. Let $T>0$, $R\in (0,1)$, $f\in L^2(\cB_R; d\mu)$, $g\in L^2((0,T);L^2(\cB_R;d\mu))$, and let $u$ be a local weak solution to equation 
\begin{equation}
\label{eq:Initial_value_problem_local_with_source}
\begin{aligned}
\left\{\begin{array}{rl}
u_t-Lu=g & \hbox{ on } (0, T)\times \cB_R,\\ 
u(0) = f& \hbox{ on } \cB_R, 
\end{array} \right.
\end{aligned}
\end{equation}
Then for all $R, T>0$, $0<r<R$, and $\eps>0$, there is a positive constant, $C=C(\eps,L,r,r_0,R,T)$, such that
\begin{equation}
\label{eq:First_order_derivatives_cutoff}
\begin{aligned}
&\sup_{s\in[t,T]}\|u(s)\varphi\|^2_{L^2(\cB_R;d\mu)}  
+\int_0^T \left(\sum_{i=1}^n\|D^{\fe_i}u(s)\varphi\|^2_{L^2(\cB_R;d\mu_{\fe_i})}
+ \sum_{l=1}^m \|D^{\ff_l}u(s)\varphi\|^2_{L^2(\cB_R;d\mu)}\right)\,ds\\
&\quad\leq 
C\left(\|u(0)\varphi\|^2_{L^2(\cB_R;d\mu)} + \int_0^T \|u(s)(\varphi+|\nabla\varphi|)\|_{L^2(\cB_R;d\mu)}\,ds\right)\\
&\quad\quad
+ \eps\int_0^T \|g(s)\varphi\|_{L^2(\cB_R;d\mu)}\,ds,
\end{aligned}
\end{equation}
where the cutoff function $\varphi$ is as in \eqref{eq:Cutoff_Sobolev_est}.
\end{lem}

\begin{rmk}
It is clear that inequality \eqref{eq:First_order_derivatives_cutoff} implies the estimate:
\begin{equation}
\label{eq:First_order_derivatives}
\begin{aligned}
&\sup_{s\in[0,T]}
\|u(s)\|^2_{L^2(\cB_r;d\mu)}
+ \int_0^T \left(\sum_{i=1}^n\|D^{\fe_i}u(s)\|^2_{L^2(\cB_r;d\mu_{\fe_i})}
+ \sum_{l=1}^m \|D^{\ff_l}u(s)\|^2_{L^2(\cB_r;d\mu)}\right)\,ds\\
&\leq C\left(\|u(0)\|^2_{L^2(\cB_R;d\mu)} + \|u\|^2_{L^2((0,T);L^2(\cB_R;d\mu))}\right) + \eps\|g\|^2_{L^2((0,T);L^2(\cB_R;d\mu))}.
\end{aligned}
\end{equation}
\end{rmk}

\begin{proof}[Proof of Lemma \ref{lem:First_order_derivatives}]
  The proof of Lemma \ref{lem:First_order_derivatives} can be done using a
  standard argument applicable to strictly elliptic differential equations,
  except that we have to be careful because the expression of the bilinear form
  $Q(u,v)$, introduced in \eqref{eq:Bilinear_form}, contains mild logarithmic
  singularities, which we show how to control. The two terms appearing in the
  expression \eqref{eq:Bilinear_form} of the bilinear form $Q(u,v)$ are given
  by the symmetric bilinear form $Q_{\hbox{\tiny{sym}}}$ defined in
  \eqref{eq:Q_sym}, and the vector field $V$ is given by \eqref{eq:V}. We
  recall that the meaning of the integer $n_0$ appearing in the expression of
  $V$ is given in \eqref{eq:n_0}. Letting $v:=\varphi^2u$ and using the
  ellipticity condition \eqref{eq:Uniform_ellipticity}, we find that there are
  positive constants, $C_1=C_1(L)$ and $C_2=C_2(L)$, such that
\begin{equation}
\label{eq:Q_sym_bound}
Q_{\hbox{\tiny{sym}}}(u, \varphi^2 u) 
\geq C_1\int_{\cB_R}\left(\sum_{i=1}^nx_i|D^{\fe_i}u|^2 +\sum_{l=1}^m|D^{\ff_l}u|^2\right)\varphi^2\,d\mu
-C_2\int_{\cB_R}\left(\varphi^2+|\nabla\varphi|^2\right)u^2\, d\mu.
\end{equation}
Using the expression of the vector field $V$, for all $\eps>0$, we have that
there is a positive constant, $C_3=C_3(\eps,L)$, such that
\begin{align*}
|(Vu,\varphi^2u)_{L^2(\cB_R;d\mu)}| 
&\leq \eps\int_{\cB_R}\left(\sum_{i=1}^nx_i|D^{\fe_i}u|^2 +\sum_{l=1}^m|D^{\ff_l}u|^2\right)\varphi^2\,d\mu\\
&\quad+C_3\int_{\cB_R}\left(\sum_{i=n_0+1}^{n}|\ln x_i|^2\right)\varphi^2u^2\, d\mu.
\end{align*}

To bound the last term in the preceding inequality, we can employ the argument
of the proof of \cite[Lemma B.3]{Epstein_Mazzeo_2016} \footnote{We note that
  the proof of \cite[Lemma B.3]{Epstein_Mazzeo_2016} requires that the weights
  $b_i\restrictedto_{\{x_i=0\}}$ are positive, for all $1\leq i\leq n$, but it
  is easy to see that the same proof holds also in the case when $V$ has
  logarithmic singularities $\ln x_i$ corresponding to weights $b_i$ such that
  $b_i\restrictedto_{\{x_i=0\}}$ is positive, which is the case in the present
  setting by definition \eqref{eq:n_0} of the integer $n_0$.}  to conclude
that, for all $\eps>0$, there is a positive constant, $C_4=C_4(\eps,L)$, with
the property that
\begin{align*}
\int_{\cB_R}\left(\sum_{i=n_0+1}^{n}|\ln x_i|^2\right)\varphi^2u^2\, d\mu
&\leq 
\eps\int_{\cB_R}\left(\sum_{i=1}^nx_i|D^{\fe_i}u|^2 +\sum_{l=1}^m|D^{\ff_l}u|^2\right)\varphi^2\,d\mu\\
&\quad + C_4\int_{\cB_R}\left(\varphi^2+|\nabla\varphi|^2\right)u^2\, d\mu.
\end{align*}
Combining the preceding two inequalities, it follows that
\begin{equation}
\label{eq:V_bound}
\begin{aligned}
|(Vu,\varphi^2u)_{L^2(\cB_R;d\mu)}| 
&\leq \eps\int_{\cB_R}\left(\sum_{i=1}^nx_i|D^{\fe_i}u|^2 +\sum_{l=1}^m|D^{\ff_l}u|^2\right)\varphi^2\,d\mu\\
&\quad+ C_5\int_{\cB_R}\left(\varphi^2+|\nabla\varphi|^2\right)u^2\, d\mu,
\end{aligned}
\end{equation}
where $C_5=C_5(\eps,L)$ is a positive constant. Choosing $\eps$ small enough, inequalities \eqref{eq:Q_sym_bound}, \eqref{eq:V_bound}, and identity \eqref{eq:Bilinear_form} yield
\begin{equation}
\label{eq:Q_bound}
\begin{aligned}
|Q(u,\varphi^2u)| 
&\geq C_1'\int_{\cB_R}\left(\sum_{i=1}^n|D^{\fe_i}u|^2 +\sum_{l=1}^m|D^{\ff_l}u|^2\right)\varphi^2\,d\mu\\
&\quad- C_2'\int_{\cB_R}\left(\varphi^2+|\nabla\varphi|^2\right)u^2\, d\mu,
\end{aligned}
\end{equation}
for some positive constants, $C'_1=C'_1(L)$ and $C'_2=C'_2(L)$. Inequality \eqref{eq:Q_bound} together with the property that
$$
\left\langle\frac{du}{dt},\varphi^2 u\right\rangle = \frac{1}{2} \frac{d}{dt}\|\varphi u\|^2_{L^2(\cB_R;d\mu)},
$$
and the fact that $u$ is a local weak solution to equation \eqref{eq:Initial_value_problem_local_with_source} give us:
\begin{align*}
&\|\varphi u(t)\|^2_{L^2(\cB_R;d\mu)} + 
C'_1\int_0^T\int_{\cB_R}\left(\sum_{i=1}^n x_i|D^{\fe_i}u(s)|^2 +\sum_{l=1}^m|D^{\ff_l}u(s)|^2\right)\varphi^2\,d\mu\,ds\\
&\quad\leq 
\|\varphi u(0)\|^2_{L^2(\cB_R;d\mu)} + C'_2 \int_0^T \int_{\cB_R}\left(\varphi^2+|\nabla\varphi|^2\right)u^2(s)\, d\mu\,ds
+ \eps\int_0^T \int_{S_{n,m}}\varphi^2g^2(s)\, d\mu\,ds,
\end{align*}
for a.e. $t\in (0,T)$, where $\eps>0$ and $C'_2=C'_2(\eps, L)$ is a positive constant. Thus, from the preceding inequality and the definition of the cutoff function $\varphi$ in \eqref{eq:Cutoff_Sobolev_est} we deduce estimate \eqref{eq:First_order_derivatives_cutoff}. This completes the proof.
\end{proof}

In the proofs of Lemma \ref{lem:Improved_first_order_derivatives} and \ref{lem:Equation_higher_order_derivatives} and of Proposition \ref{prop:Higher_order_derivatives}, we work with smooth functions, $u\in C^{\infty}([0,T]\times\bar\cB_R)$, but this strong hypothesis will be removed in the proof of Theorem \ref{thm:Regularity_Sobolev} via an approximation argument of local weak solutions by smooth functions established in Lemma \ref{lem:Approx_smooth}.

\begin{lem}[Improved first- and second-order derivatives estimates]
\label{lem:Improved_first_order_derivatives}
Assume that the hypotheses of Theorem \ref{thm:Regularity_Sobolev} hold and that the local weak solution $u$ belongs to 
$C^{\infty}([0,T]\times \bar \cB_R)$ and that, for all $(\fa,\fb)\in\NN^n\times\NN^m$ such that $|\fa|+|\fb|=1$, and for all $0<T_0<T$, the derivative $D^{\fa}_xD^{\fb}_yu$ belongs to $L^2((T_0,T);H^1(B_R;d\mu_{\fa}))$. Then for all $0<t<T$ and $0<r<R$, there is a positive constant, 
$C=C(L,r,R,t,T)$, such that
\begin{equation}
\label{eq:Improved_first_order_derivatives}
\begin{aligned}
&\sup_{s\in[t,T]}
\left(\sum_{i=1}^n\|D^{\fe_i}u(s)\|^2_{L^2(\cB_r;d\mu_{\fe_i})} + \sum_{l=1}^m\|D^{\ff_l}u(s)\|^2_{L^2(\cB_r;d\mu)}\right)\\
&+\int_t^T
\left(\sum_{i,j=1}^n\|D^{\fe_i+\fe_j}u(s)\|^2_{L^2(\cB_r;d\mu_{\fe_i+\fe_j})}
+\sum_{i=1}^n\sum_{l=1}^m\|D^{\fe_i+\ff_l}u(s)\|^2_{L^2(\cB_r;d\mu_{\fe_i})}\right.\\
&\quad\quad\left.
+\sum_{l,k=1}^m\|D^{\ff_l+\ff_k}u(s)\|^2_{L^2(\cB_r;d\mu)}\right)\,ds\\
&\leq C\left(\|u(0)\|^2_{L^2(\cB_R;d\mu)} + \|u\|^2_{L^2((0,T);L^2(\cB_R;d\mu))}\right).
\end{aligned}
\end{equation}
\end{lem}

\begin{proof}
Because the weak solution is assumed to be a smooth function on $[0,\infty)\times \bar \cB_R$, we can take derivatives in the equation satisfied by $u$ and use estimates \eqref{eq:First_order_derivatives} applied to the derivatives of $u$ instead of $u$ to derive \eqref{eq:Improved_first_order_derivatives}.

\setcounter{step}{0}
\begin{step}[Derivatives in the $x$-variables]
\label{step:First_derivative_x}
Let $1\leq i_0\leq n$. Our goal is to prove that for all $\eps>0$, there is a positive constant, $C=C(\eps)$, such that
\begin{equation}
\label{eq:After_integration_x}
\begin{aligned}
&\|D^{\fe_{i_0}}u(T)\varphi\|^2_{L^2(\cB_R;d\mu_{\fe_{i_0}})} 
+ \int_t^T\int_{\cB_R}\left(\sum_{i=1}^n x_i|D^{\fe_{i_0}+\fe_i}u(s)|^2 +\sum_{l=1}^m|D^{\fe_{i_0}+\ff_l}u(s)|^2\right)\varphi^2
\,d\mu_{\fe_{i_0}}\, ds
\\
&\leq C\left(\|D^{\fe_{i_0}}u(t)\varphi\|^2_{L^2(\cB_R;d\mu_{\fe_{i_0}})} 
+\|u(0)\|^2_{L^2(\cB_R;d\mu)} + \|u\|^2_{L^2((0,T);L^2(\cB_r;d\mu))}\right)\\
&\quad+\eps\int_t^T\int_{\cB_R}\left(\sum_{i,j=1}^n|x_ix_jD^{\fe_i+\fe_j}u(s)|^2 
+ \sum_{i=1}^n\sum_{l=1}^m|x_iD^{\fe_i+\ff_l}u(s)|^2 \right)  \varphi^2\,d\mu_{\fe_{i_0}}\, ds\\
&\quad+\eps\int_t^T\int_{\cB_R}  \sum_{l,k=1}^m|D^{\ff_l+\ff_k}u(s)|^2 \varphi^2\,d\mu_{\fe_{i_0}}\, ds,
\end{aligned}
\end{equation}
where the cutoff function $\varphi$ is as in \eqref{eq:Cutoff_Sobolev_est}. Taking a derivative in the $x_{i_0}$-variable in equation 
$u_t-Lu = 0$ on $(0,T)\times\cB_R$, we see that the first-order derivative $D^{\fe_{i_0}}u$ satisfies an equation
\begin{equation}
\label{eq:Equation_x}
\left(\partial_t- L^{\fe_{i_0}}\right) D^{\fe_{i_0}}u = g^{\fe_{i_0}},
\end{equation}
where the operator $L^{\fe_{i_0}}$ is also of the form \eqref{eq:Operator} with $\bar a_{ii} =1$ and with weights $b_i$ replaced by 
$b_i':=b_i+\delta_{i_0,i}$, for all $1\leq i\leq n$, and so the measure associated to the operator $L^{\fe_{i_0}}$ using the weights 
$\{b'_i:1\leq i\leq n\}$ in identity \eqref{eq:Weight}, where we replace $b_i$ by $b'_i$, is exactly $d\mu_{\fe_{i_0}}$.

The source function $g^{\fe_{i_0}}$ on the right-hand side of \eqref{eq:Equation_x} is a linear combination with smooth coefficients of at most derivatives of the form:
$x_ix_jD^{\fe_i+\fe_j}u$, $x_iD^{\fe_i+\ff_l}u$, $D^{\ff_l+\ff_k}u$, $\partial_{x_{i_0}}b_i D^{\fe_i}u$, $D^{\ff_l}u$, $u$,
for all $1\leq i,j\leq n$ and for all $1\leq l,k\leq m$. For all $1\leq i\leq n_0$, where the integer $n_0$ is defined in \eqref{eq:n_0}, we have that $b_i$ restricted to $\{x_i=0\}$ is a constant function, and so the smoothness of the coefficients of the operator $L$ imply that 
\begin{equation}
\label{eq:Important_property}
|\partial_{x_{i_0}}b_i| \leq cx_i,\quad\forall\, 1\leq i\leq n_0\hbox{ such that } i\neq i_0,
\end{equation}
where $c$ is a positive constant. Thus, we can rewrite the source function $g^{\fe_{i_0}}$ as a a linear combination with smooth coefficients of derivatives of the form:
\begin{equation}
\label{eq:Derivatives_source_x}
\begin{aligned}
&x_ix_jD^{\fe_i+\fe_j}u,\quad x_iD^{\fe_i+\ff_l}u,\quad D^{\ff_l+\ff_k}u,\quad D^{\ff_l}u,\quad u,\\
&x_iD^{\fe_i}u,\quad\hbox{ if } 1\leq i\leq n_0,\quad\hbox{ and } \quad D^{\fe_i}u,\quad\hbox{ if } n_0+1\leq i\leq n, 
\end{aligned}
\end{equation}
for all $1\leq i,j\leq n$ and for all $1\leq l,k\leq m$.  Estimate \eqref{eq:First_order_derivatives_cutoff} applied to $D^{\fe_{i_0}}u$ yields:
\begin{align*}
&\|D^{\fe_{i_0}}u(T)\varphi\|^2_{L^2(\cB_R;d\mu_{\fe_{i_0}})} 
+ \int_t^T\int_{\cB_R}\left(\sum_{i=1}^n x_i|D^{\fe_{i_0}+\fe_i}u(s)|^2 +\sum_{l=1}^m|D^{\fe_{i_0}+\ff_l}u(s)|^2\right)\varphi^2
\,d\mu_{\fe_{i_0}}\, ds
\\
&\leq C\left(\|D^{\fe_{i_0}}u(t)\varphi\|^2_{L^2(\cB_R;d\mu_{\fe_{i_0}})} 
+ \int_t^T\|D^{\fe_{i_0}}u(s)\|^2_{L^2(\cB_{r_0};d\mu_{\fe_{i_0}})}\, ds\right)\\
&\quad+
\eps\int_t^T\int_{\cB_R}\left(\sum_{i,j=1}^n x_ix_j|D^{\fe_i+\fe_j}u(s)|^2 
+ \sum_{i=1}^n\sum_{l=1}^m x_i|D^{\fe_i+\ff_l}u(s)|^2
\right) \varphi^2\,d\mu_{\fe_{i_0}}\, ds\\
&\quad+\eps\int_t^T\int_{\cB_R}\sum_{l,k=1}^m  |D^{\ff_l+\ff_k}u(s)|^2\varphi^2\,d\mu_{\fe_{i_0}}\, ds\\
&\quad
+ \eps\int_t^T\int_{\cB_R}\left(\sum_{i=1}^n (x_i\mathbf{1}_{\{i\leq n_0\}}+\mathbf{1}_{\{i> n_0\}})|D^{\fe_i}u(s)|^2
+\sum_{l=1}^m|D^{\ff_l}u(s)|^2 + |u(s)|^2\right) \varphi^2\,d\mu_{\fe_{i_0}}\, ds.
\end{align*}
We can now use the local estimate \eqref{eq:First_order_derivatives} applied to the second term and to the last term containing indices $1\leq i\leq n_0$ on the right-hand side of the preceding inequality, to obtain that
\begin{equation}
\label{eq:After_integration_x_2}
\begin{aligned}
&\|D^{\fe_{i_0}}u(T)\varphi\|^2_{L^2(\cB_R;d\mu_{\fe_{i_0}})} 
+ \int_t^T\int_{\cB_R}\left(\sum_{i=1}^n x_i|D^{\fe_{i_0}+\fe_i}u(s)|^2 +\sum_{l=1}^m|D^{\fe_{i_0}+\ff_l}u(s)|^2\right)\varphi^2
\,d\mu_{\fe_{i_0}}\, ds
\\
&\leq C\left(\|D^{\fe_{i_0}}u(t)\varphi\|^2_{L^2(\cB_R;d\mu_{\fe_{i_0}})} 
+\|u(0)\|^2_{L^2(\cB_R;d\mu)} + \|u\|^2_{L^2((0,T);L^2(\cB_R;d\mu))}\right)\\
&\quad
+\eps\int_t^T\int_{\cB_R}\left(\sum_{i,j=1}^nx_ix_j|D^{\fe_i+\fe_j}u(s)|^2 
+ \sum_{i=1}^n\sum_{l=1}^m x_i|D^{\fe_i+\ff_l}u(s)|^2 \right)  \varphi^2\,d\mu_{\fe_{i_0}}\, ds\\
&\quad+\eps\int_t^T\int_{\cB_R}  \sum_{l,k=1}^m|D^{\ff_l+\ff_k}u(s)|^2 \varphi^2\,d\mu_{\fe_{i_0}}\, ds
+\eps\int_t^T\int_{\cB_R}\sum_{i=n_0+1}^n  |D^{\fe_i}u(s)|^2 \varphi^2\,d\mu_{\fe_{i_0}}\, ds.
\end{aligned}
\end{equation}
To obtain a bound on the last term on the right-hand side of the preceding inequality we apply inequality 
\eqref{eq:Bad_first_derivative} with $u$ replaced by $D^{\fe_i}u$ and with the weight $d\mu$ replaced by $d\mu_{\fe_{i_0}}$. This gives us that, for all $ i\neq i_0$, we have
\begin{align*}
\int_t^T\int_{\cB_R}  |D^{\fe_i}u(s)|^2 \varphi^2\,d\mu_{\fe_{i_0}}\, ds
&\leq
\int_t^T\int_{\cB_R}\left(\sum_{j=1}^nx_j|D^{\fe_i+\fe_j}u|^2 +\sum_{l=1}^m|D^{\fe_i+\ff_l}u|^2\right)\varphi^2\,d\mu_{\fe_i+\fe_{i_0}}\, ds\\
&\quad + C\int_t^T\int_{\cB_R}\left(\varphi^2+|\nabla\varphi|^2\right)|D^{\fe_i}u|^2\, d\mu_{\fe_i+\fe_{i_0}}\, ds.
\end{align*}
Applying \eqref{eq:First_order_derivatives_cutoff} to the last term on the right-hand side above, it follows that
\begin{align*}
\int_t^T\int_{\cB_R}  |D^{\fe_i}u(s)|^2 \varphi^2\,d\mu_{\fe_{i_0}}\, ds
&\leq
\int_t^T\int_{\cB_R}\left(\sum_{j=1}^nx_j|D^{\fe_i+\fe_j}u|^2 +\sum_{l=1}^m|D^{\fe_i+\ff_l}u|^2\right)\varphi^2\,d\mu_{\fe_i+\fe_{i_0}}\, ds\\
&\quad + C\left(\|u(0)\|^{L^2(\cB_R;d\mu)} + \|u\|^2_{L^2((0,T);L^2(\cB_R;d\mu))}\right).
\end{align*}
The preceding inequality applied for all $i\neq u_0$, together with \eqref{eq:First_order_derivatives} applied to the last term in \eqref{eq:After_integration_x_2} containing the derivative $D^{\fe_{i_0}}$, and inequality \eqref{eq:After_integration_x_2} yield \eqref{eq:After_integration_x}. This completes the proof of Step \ref{step:First_derivative_x}.
\end{step}

\begin{step}[Derivatives in the $y$-variables]
\label{step:First_derivative_y}
Let $1\leq l_0\leq m$. Our goal is now to prove that for all $\eps>0$, there is a positive constant, $C=C(\eps,L,r,R_0,R)$, such that
\begin{equation}
\label{eq:After_integration_y}
\begin{aligned}
&\|D^{\ff_{l_0}}u(T)\varphi\|^2_{L^2(\cB_r;d\mu)} 
+ \int_t^T\int_{\cB_R}\left(\sum_{i=1}^n x_i|D^{\ff_{l_0}+\fe_i}u(s)|^2 +\sum_{l=1}^m|D^{\ff_{l_0}+\ff_l}u(s)|^2\right)\varphi^2
\,d\mu\, ds
\\
&\leq C\left(\|D^{\ff_{l_0}}u(t)\varphi\|^2_{L^2(\cB_R;d\mu)} + \|u(0)\|^2_{L^2(\cB_R;d\mu)} + \|u\|^2_{L^2((0,T);L^2(\cB_R;d\mu))}\right)\\
&\quad+\eps\int_t^T\int_{\cB_R}\left(\sum_{i,j=1}^n x_ix_j|D^{\fe_i+\fe_j}u(s)|^2 
+ \sum_{i=1}^n\sum_{l=1}^mx_i|D^{\fe_i+\ff_l}u(s)|^2\right) \varphi^2\,d\mu\, ds\\
&\quad+\eps\int_t^T\int_{\cB_R} \sum_{l,k=1}^m|D^{\ff_l+\ff_k}u(s)|^2 \varphi^2\,d\mu\, ds.
\end{aligned}
\end{equation}
Taking a derivative in the $y_{l_0}$-variable in equation $u_t-Lu=0$ on $(0,T)\times\cB_R$, we see that the first-order derivative $D^{\ff_{l_0}}u$ satisfies an equation
\begin{equation}
\label{eq:Equation_y}
\left(\partial_t- L^{\ff_{l_0}}\right) D^{\fl_{l_0}}u = g^{\ff_{i_0}},
\end{equation}
where the operator $L^{\ff_{l_0}}$ has the same form as $L$ is \eqref{eq:Operator} and the same weights as $L$, and the source function $g^{\ff_{l_0}}$ is a linear combination with smooth coefficients of the derivatives \eqref{eq:Derivatives_source_x}, for all $1\leq i,j\leq n$ and for all $1\leq l,k\leq m$ such that $l\neq l_0$. From here on we can apply the same argument as the one used in the proof of estimate \eqref{eq:After_integration_x} with the modification that we use the measure $d\mu$ instead of $d\mu_{\fe_{i_0}}$. We omit the detailed proof as it is identical to that of Step \ref{step:First_derivative_x}. This completes the proof of Step \ref{step:First_derivative_y}.
\end{step}

Adding inequalities \eqref{eq:After_integration_x} applied to all $1\leq i_0\leq n$ and \eqref{eq:After_integration_y} applied to all $1\leq l_0\leq m$, and choosing $\eps:=1/2$, we obtain for all $t\in (0,T)$,
\begin{align*}
&\sum_{i=1}^n\|D^{\fe_i}u(T)\|^2_{L^2(\cB_r;d\mu_{\fe_i})} + \sum_{l=1}^m\|D^{\ff_l}u(T)\|^2_{L^2(\cB_r;d\mu)}\\
&+\int_t^T
\left(\sum_{i,j=1}^n\|D^{\fe_i+\fe_j}u(s)\|^2_{L^2(\cB_r;d\mu_{\fe_i+\fe_j})}
+\sum_{i=1}^n\sum_{l=1}^m\|D^{\fe_i+\ff_l}u(s)\|^2_{L^2(\cB_r;d\mu_{\fe_i})}\right.\\
&\quad\quad\left.
+\sum_{l,k=1}^m\|D^{\ff_l+\ff_k}u(s)\|^2_{L^2(\cB_r;d\mu)}\right)\,ds\\
&\leq C\left(\|u(0)\|^2_{L^2(B_R;d\mu)} +\|u\|^2_{L^2((0,T);L^2(\cB_R;d\mu))}\right)\\
&\quad\quad + C\left(\sum_{i=1}^n \|D^{\fe_i}u(t)\|^2_{L^2(\cB_{r_0};d\mu_{\fe_i})}
+\sum_{l=1}^m \|D^{\ff_l}u(t)\|^2_{L^2(\cB_{r_0};d\mu)}\right).
\end{align*}
Integrating the preceding inequality in $t$ from 0 to $T$ and using estimate \eqref{eq:First_order_derivatives}, it follows that
\begin{align*}
&\sum_{i=1}^nT\|D^{\fe_i}u(T)\|^2_{L^2(\cB_r;d\mu_{\fe_i})} + \sum_{l=1}^mT\|D^{\ff_l}u(T)\|^2_{L^2(\cB_r;d\mu)}\\
&+\int_0^T
\left(\sum_{i,j=1}^n\|D^{\fe_i+\fe_j}u(s)\|^2_{L^2(\cB_r;d\mu_{\fe_i+\fe_j})}
+\sum_{i=1}^n\sum_{l=1}^m\|D^{\fe_i+\ff_l}u(s)\|^2_{L^2(\cB_r;d\mu_{\fe_i})}\right.\\
&\quad\quad\left.
+\sum_{l,k=1}^m\|D^{\ff_l+\ff_k}u(s)\|^2_{L^2(\cB_r;d\mu)}\right)\,sds\\
&\leq C\left(\|u(0)\|^2_{L^2(B_R;d\mu)} +\|u\|^2_{L^2((0,T);L^2(\cB_R;d\mu))}\right).
\end{align*}
From here estimate \eqref{eq:Improved_first_order_derivatives} follows immediately. This completes the proof.
\end{proof}

We next establish the form of the equation satisfied by the higher-order derivatives of the solution $u$. We assume that $\NN^n$ and $\NN^m$ are ordered lexicographically.

\begin{lem}[Equation for higher-order derivatives]
\label{lem:Equation_higher_order_derivatives}
Assume that the hypotheses of Theorem \ref{thm:Regularity_Sobolev} hold and that the local weak solution $u$ belongs to 
$C^{\infty}([0,T]\times \bar \cB_R)$. Then for all $(\fa,\fb)\in\NN^n\times\NN^m$ such that $(\fa,\fb)\neq(0,0)$, the derivative $D^{\fa}_xD^{\fb}_yu$ is a solution to
\begin{equation}
\label{eq:Equation_higher_derivative}
\left(\partial_t-L^{(\fa,\fb)}\right) D^{\fa}_xD^{\fb}_yu = g^{(\fa,\fb)},\quad\hbox{ on } (0,T)\times\cB_R,
\end{equation}
where $L^{(\fa,\fb)}$ is an operator that has the same structure as $L$ defined in \eqref{eq:Operator} and its coefficients satisfy Assumption \ref{assump:Coeff}. Moreover, the weights of the operator $L^{(\fa,\fb)}$ are $\{b_i+\fa_i:1\leq i\leq n\}$ and the source function, $g^{(\fa,\fb)}$, contains linear combinations with smooth coefficients of at most the following derivatives of the function $u$:
\begin{equation}
\label{eq:Source_function_higher}
\begin{aligned}
&x_ix_jD^{\fa'+\fe_i+\fe_j}_xD^{\fb''}_yu,\quad x_iD^{\fa'+\fe_i}_xD^{\fb''+\ff_l}_yu,\quad D^{\fa'}_xD^{\fb''+\ff_l}_yu,
\quad D^{\fa'}_xD^{\fb''+\ff_l+\ff_k}_yu,\quad D^{\fa'}_xD^{\fb''}_yu,\\
&x_ix_jD^{\fa''+\fe_i+\fe_j}_xD^{\fb'}_yu,\quad x_iD^{\fa''+\fe_i}_xD^{\fb'+\ff_l}_yu,\quad D^{\fa''}_xD^{\fb'+\ff_l}_yu,
\quad D^{\fa''}_xD^{\fb'+\ff_l+\ff_k}_yu,\quad D^{\fa''}_xD^{\fb'}_yu,\\
&x_i D^{\fa'+\fe_i}_xD^{\fb''}_yu,\quad x_i D^{\fa''+\fe_i}_xD^{\fb'}_yu,\quad\hbox{ if } 1\leq i\leq n_0,\\
& D^{\fa'+\fe_i}_xD^{\fb'}_yu,\quad D^{\fa''+\fe_i}_xD^{\fb'}u,\quad\hbox{ if } n_0+1\leq i\leq n, 
\end{aligned}
\end{equation}
for all $\fa',\fa''\in\NN^n$ and $\fb',\fb''\in\NN^m$, such that $\fa'<\fa$, $\fa''\leq\fa$, $\fb'<\fb$, $\fb''\leq\fb$, and for all $1\leq i,j\leq n$ and $1\leq l,k\leq m$. 
\end{lem}

\begin{proof}
We argue by induction on $N:=|\fa|+|\fb|\geq 1$. When $N=1$, the statement in Lemma \ref{lem:Equation_higher_order_derivatives} is proved in Step \ref{step:First_derivative_x} of Lemma \ref{lem:Improved_first_order_derivatives}, when $\fa=\fe_i$ and $1\leq i\leq n$, and in Step \ref{step:First_derivative_y} of Lemma \ref{lem:Improved_first_order_derivatives}, when $\fb=\ff_l$ and $1\leq l\leq m$. For the induction step $N\geq 1$, we fix $(\fa,\fb)\in\NN^n\times\NN^m$ such that $|\fa|+|\fb|=N$ and we assume that \eqref{eq:Equation_higher_derivative} holds, where the function $g^{(\fa,\fb)}$ can be written as a linear combination with smooth coefficients of the functions enumerated in \eqref{eq:Source_function_higher}. We want to prove the analogous statement for $D^{\fa+\fe_{i_0}}_xD^{\fb}_y$ and $D^{\fa}_xD^{\fb+\ff_{l_0}}_y$, for all $1\leq i_0\leq n$ and for all $1\leq l_0\leq m$. We use the following observations. We have that
\begin{equation}
\label{eq:Derivative_a_b}
\begin{aligned}
D^{\fe_{i_0}}_x\left(\left(\partial_t- L^{(\fa,\fb)}\right) D^{\fa}_xD^{\fb}_yu\right) 
&= \left(\partial_t- L^{(\fa+\fe_{i_0},\fb)}\right) D^{\fa+\fe_{i_0}}_xD^{\fb}_yu + h^{(\fa,\fb)+\fe_{i_0}},\\ 
D^{\ff_{l_0}}_y\left(\left(\partial_t- L^{(\fa,\fb)}\right) D^{\fa}_xD^{\fb}_yu\right) 
&= \left(\partial_t- L^{(\fa,\fb+\ff_{l_0})}\right) D^{\fa}_xD^{\fb+\ff_{l_0}}_yu + h^{(\fa,\fb)+\ff_{l_0}}, 
\end{aligned}
\end{equation}
where to obtain the form of the functions $h^{(\fa,\fb)+\fe_{i_0}}$ and $h^{(\fa,\fb)+\ff_{l_0}}$ we apply \eqref{eq:Derivatives_source_x} with $u$ replaced by $D^{\fa}_xD^{\fb}_yu$. We note that the argument used to prove \eqref{eq:Derivatives_source_x} with $u$ adapts immediately to $D^{\fa}_xD^{\fb}_yu$ because it only uses the fact that, for all $1\leq i\leq n_0$, the weight $b_i$ of the operator $L$ is constant along $\{x_i=0\}$, which is also true for the weight $b_i+\fa_i$ of the operator $L^{(\fa,\fb)}$. Thus, we obtain that the functions $h^{(\fa,\fb)+\fe_{i_0}}$ and $h^{(\fa,\fb)+\ff_{l_0}}$ are linear combinations with smooth coefficients of at most the following derivatives:
\begin{equation}
\label{eq:Source_function_higher_h}
\begin{aligned}
&x_ix_jD^{\fa+\fe_i+\fe_j}_xD^{\fb}_yu,\quad x_iD^{\fa+\fe_i}_xD^{\fb+\ff_l}_yu,\quad D^{\fa}_xD^{\fb+\ff_l+\ff_k}_yu,
\quad D^{\fa}_xD^{\fb+\ff_l}_yu,\quad D^{\fa}_xD^{\fb}_yu,\\
&x_iD^{\fa+\fe_i}D^{\fb}_yu,\quad\hbox{ if } 1\leq i\leq n_0,
\quad\hbox{ and }\quad
D^{\fa+\fe_i}_xD^{\fb}_yu,\quad\hbox{ if } n_0+1\leq i\leq n.
\end{aligned}
\end{equation}
for all $1\leq i,j\leq n$ and $1\leq l,k\leq m$. From \eqref{eq:Derivative_a_b} and \eqref{eq:Equation_higher_derivative}, we obtain that
\begin{align*}
\left(\partial_t- L^{(\fa+\fe_{i_0},\fb)}\right) D^{\fa+\fe_{i_0}}_xD^{\fb}_yu = D^{\fe_{i_0}}_xg^{(\fa,\fb)} - h^{(\fa+\fe_{i_0},\fb)},\\ 
\left(\partial_t- L^{(\fa,\fb+\ff_{l_0})}\right) D^{\fa}_xD^{\fb+\ff_{l_0}}_yu = D^{\ff_{l_0}}_y g^{(\fa,\fb)} - h^{(\fa,\fb+\ff_{l_0})},
\end{align*}
from where we see that
\begin{align*}
g^{(\fa+\fe_{i_0},\fb)}=D^{\fe_{i_0}}_xg^{(\fa,\fb)} - h^{(\fa+\fe_{i_0},\fb)},
\quad\quad
g^{(\fa,\fb+\ff_{l_0})}=D^{\ff_{l_0}}_yg^{(\fa,\fb)} - h^{(\fa,\fb)+\ff_{l_0}}.
\end{align*}
Using the fact that $g^{\fe_{i}}$ and $g^{\ff_l}$ is a linear combination with smooth coefficients of the derivatives in \eqref{eq:Derivatives_source_x} and that 
$h^{(\fa+\fe_{i_0},\fb)}$ and $h^{(\fa,\fb+\ff_{l_0})}$ are linear combinations of the terms in \eqref{eq:Source_function_higher_h}, we can use the preceding identities to prove inductively the form \eqref{eq:Source_function_higher} of the source functions 
$g^{(\fa+\fe_{i_0},\fb)}$ and $g^{(\fa,\fb+\ff_{l_0})}$. We omit the details as they are very tedious, though elementary to establish. This completes the proof.
\end{proof}

We now use Lemma \ref{lem:First_order_derivatives}, \ref{lem:Improved_first_order_derivatives}, and \ref{lem:Equation_higher_order_derivatives} to prove a weaker version of Theorem \ref{thm:Regularity_Sobolev}, in which we assume that the local weak solution is smooth. 

\begin{prop}[Higher-order derivative estimates]
\label{prop:Higher_order_derivatives}
Assume that the hypotheses of Theorem \ref{thm:Regularity_Sobolev} hold and that the local weak solution $u$ belongs to $C^{\infty}([0,\infty)\times \bar \cB_R)$ and that, for all $(\fa,\fb)\in\NN^n\times\NN^m$ and for all $0<T_0<T$, the derivative $D^{\fa}_xD^{\fb}_yu$ belongs to $L^2((T_0,T);H^1(B_R;d\mu_{\fa}))$. Then for all $0<t<T$ and $0<r<R<1$ there is a positive constant, 
$C=C(\fa,\fb,L,r,R,t,T)$, such that
\begin{equation}
\label{eq:Higher_order_derivatives}
\begin{aligned}
&\sup_{s\in[t,T]}
\|D^{\fa}_xD^{\fb}_yu(s)\|^2_{L^2(\cB_r;d\mu_{\fa})}\\
&+ \int_t^T \left(\sum_{i=1}^n\|D^{\fa+\fe_i}_xD^{\fb}_yu(s)\|^2_{L^2(\cB_r;d\mu_{\fa+\fe_i})}
+ \sum_{l=1}^m \|D^{\fa}_xD^{\fb)+\ff_l}_yu(s)\|^2_{L^2(\cB_r;d\mu_{\fa})}\right)\,ds\\
&\leq C\left(\|u(0)\|^2_{L^2(\cB_R;d\mu)}+\|u\|^2_{L^2((0,T);L^2(\cB_R;d\mu))}\right).
\end{aligned}
\end{equation}
\end{prop}

\begin{proof}
By Lemma \ref{lem:Equation_higher_order_derivatives}, we know that $D^{\fa}_xD^{\fb}_yu$ is a local weak solution to equation 
\eqref{eq:Equation_higher_derivative}, and so we can apply estimate \eqref{eq:First_order_derivatives_cutoff} to $D^{\fa}_xD^{\fb}_yu$ on the interval $[t,T]$, with $0<t<T$, to obtain that
\begin{equation}
\label{eq:Eq_for_induction_1}
\begin{aligned}
&\|D^{\fa}_xD^{\fb}_yu(T)\varphi\|^2_{L^2(\cB_R;d\mu_{\fa})}\\
&+ \int_t^T \left(\sum_{i=1}^n\|D^{\fa+\fe_i}_xD^{\fb}_yu(s)\varphi\|^2_{L^2(\cB_R;d\mu_{\fa+\fe_i})}
+ \sum_{l=1}^m \|D^{\fa}_xD^{\fb+\ff_l}_yu(s)\varphi\|^2_{L^2(\cB_R;d\mu_{\fa})}\right)\,ds\\
&\leq C\left(\|D^{\fa}_xD^{\fb}_yu(t)\varphi\|^2_{L^2(\cB_R;d\mu_{\fa})} 
+ \int_t^T \|D^{\fa}_xD^{\fb}_yu(s)\|^2_{L^2(\cB_{r_0};d\mu_{\fa})}\, ds\right)\\
&\quad + \eps\int_t^T \int_{\cB_R} |g^{(\fa,\fb)}(s)|^2\varphi^2 \,d\mu_{\fa}\,ds,
\end{aligned}
\end{equation}
for all $\eps>0$. For all $N\geq 1$, we let
\begin{equation}
\label{eq:Definition_I_N}
\begin{aligned}
I_N&:= \sum_{\stackrel{\fa\in\NN^n,\,\fb\in\NN^m}{|\fa|+|\fb|\leq N}} 
\int_t^T \left(\sum_{i=1}^n\|D^{\fa+\fe_i}_xD^{\fb}_yu(s)\varphi\|^2_{L^2(\cB_R;d\mu_{\fa+\fe_i})}\right.\\
&\qquad\qquad\left.
+ \sum_{l=1}^m \|D^{\fa}_xD^{\fb+\ff_l}_yu(s)\varphi\|^2_{L^2(\cB_R;d\mu_{\fa})}\right)\,ds.
\end{aligned}
\end{equation}
We will prove by induction on $N$ that
\begin{equation}
\label{eq:Induction_I_N}
I_N \leq C \left(\|u(0)\|^2_{L^2(\cB_R;d\mu)} + \|u\|^2_{L^2((0,T);L^2(\cB_R;d\mu))}\right),
\end{equation}
where $C=C(N,L,r,R,t,T)$ is a positive constant. When $N=0$, inequality \eqref{eq:Induction_I_N} follows immediately from \ref{eq:First_order_derivatives_cutoff}. Assume that inequality \eqref{eq:Induction_I_N} holds for $N-1$. We want to prove that it also holds for $N$. Let $(\fa,\fb)\in\NN^n\times\NN^m$ be such that $|\fa|+|\fb|=N$. We write the source function $g^{(\fa,\fb)}=g^{(\fa,\fb)}_1+g^{(\fa,\fb)}_2$, where the summands are chosen as follows based on the fact that $g^{(\fa,\fb)}$ is a linear combination with smooth coefficients of the functions enumerated in \eqref{eq:Source_function_higher}. We choose the function $g^{(\fa,\fb)}_1$ such that it contains a linear combination with smooth coefficients of the derivatives:
\begin{align*}
&x_ix_jD^{\fa'+\fe_i+\fe_j}_xD^{\fb''}_yu,\quad x_iD^{\fa'+\fe_i}_xD^{\fb''+\ff_l}_yu,
\quad D^{\fa'}_xD^{\fb''+\ff_l+\ff_k}_yu,\quad D^{\fa'}_xD^{\fb''+\ff_l}_yu,\quad D^{\fa'}_xD^{\fb''}_yu,\\
&x_ix_jD^{\fa''+\fe_i+\fe_j}_xD^{\fb'}_yu,\quad x_iD^{\fa''+\fe_i}_xD^{\fb'+\ff_l}_yu,
\quad D^{\fa''}_xD^{\fb'+\ff_l+\ff_k}_yu,\quad D^{\fa''}_xD^{\fb'+\ff_l}_yu,\quad D^{\fa''}_xD^{\fb'}_yu,\\
&x_i D^{\fa'+\fe_i}_xD^{\fb''}u,\quad x_i D^{\fa''+\fe_i}_xD^{\fb'}_yu,\quad\hbox{ if } 1\leq i\leq n_0. 
\end{align*}
We choose the function $g^{(\fa,\fb)}_2$ such that it contains a linear combination with smooth coefficients of the derivatives:
$$
D^{\fa'+\fe_i}_xD^{\fb''}u,\quad D^{\fa''+\fe_i}_xD^{\fb'}_yu,\quad\hbox{ if } n_0+1\leq i\leq n.
$$
In the expressions of the functions $g^{(\fa,\fb)}_i$, for $i=1,2$, we assume that $\fa'<\fa$, $\fa''\leq\fa$, $\fb'<\fb$, and $\fb''\leq \fb$. Note that from the definition of $g_1^{(\fa,\fb)}$ and of $I_N$ in \eqref{eq:Induction_I_N}, there is a positive constant, 
$C=C(L,N,m,n)$, such that
\begin{equation}
\label{eq:Summand_1}
\int_t^T \int_{\cB_R} |g^{(\fa,\fb)}_1(s)|^2\varphi^2 \,d\mu_{\fa}\,ds \leq C I_{N-1}.
\end{equation}
We next estimate the term $g^{(a,b)}_2$. Let $i$ be such that $n_0+1\leq i\leq n$ and $\fa'+\fe_i\neq \fa$. Applying inequality \eqref{eq:Bad_first_derivative} to $D^{\fa'+\fe_i}_xD^{\fb''}_yu$ instead of $u$, with $d\mu$ replaced by $d\mu_{\fa}$, gives us that
\begin{align*}
\int_{\cB_R}  |D^{\fa'+\fe_i}_xD^{\fb''}_yu(s)|^2 \varphi^2\,d\mu_{\fa}
&\leq 
\int_{\cB_R}\left(\sum_{j=1}^nx_j|D^{\fa'+\fe_i+\fe_j}_xD^{\fb''}_yu(s)|^2 
+\sum_{l=1}^m|D^{\fa'+\fe_i}_xD^{\fb''+\ff_l}_yu(s)|^2\right)\varphi^2\,d\mu_{\fa+\fe_i}\\
&\quad + C\int_{\cB_R}\left(\varphi^2+|\nabla\varphi|^2\right)|D^{\fa'+\fe_i}_xD^{\fb''}_yu(s)|^2\, d\mu_{\fa+\fe_i},
\end{align*}
and the induction hypothesis \eqref{eq:Induction_I_N} applied to the right-hand side of the preceding inequality yields
\begin{align*}
\int_t^T\int_{\cB_R}  |D^{\fa'+\fe_i}_xD^{\fb''}_yu(s)|^2 \varphi^2\,d\mu_{\fa}\,ds
&\leq 
C \left(\|u(0)\|^2_{L^2(\cB_R;d\mu)} + \|u\|^2_{L^2((0,T);L^2(\cB_R;d\mu))}\right).
\end{align*}
When $i$ is such that $n_0+1\leq i\leq n$ and $\fa'+\fe_i = \fa$, it follows from the definition of $I_N$ in \eqref{eq:Induction_I_N} that
\begin{align*}
\int_t^T\int_{\cB_R}  |D^{\fa'+\fe_i}_xD^{\fb''}_yu(s)|^2 \varphi^2\,d\mu_{\fa}\,ds
&\leq CI_{N-1}.
\end{align*}
The preceding argument applied to $D^{\fa'+\fe_i}_xD^{\fb''}_yu$ can also be applied to $DD^{\fa''+\fe_i}_xD^{\fb'}_yu$ in order to obtain the preceding inequality with $D^{\fa'+\fe_i}_xD^{\fb''}_yu$ replaced by $D^{\fa''+\fe_i}_xD^{\fb'}_yu$ on the left-hand side. Hence, using the expression of the function $g^{(\fa,\fb)}_3$, we obtain that
\begin{equation}
\label{eq:Summand_3}
\int_t^T\int_{\cB_R}  |g^{(\fa,\fb)}_3 u|^2 \varphi^2\,d\mu_{\fa}\,ds
\leq 
C \left(\|u(0)\|^2_{L^2(\cB_R;d\mu)} + \|u\|^2_{L^2((0,T);L^2(\cB_R;d\mu))}\right).
\end{equation}
Using estimates \eqref{eq:Summand_1}, and \eqref{eq:Summand_3} in \eqref{eq:Eq_for_induction_1}, together with definition \eqref{eq:Definition_I_N} of $I_N$, it follows that
\begin{align*}
\sum_{\stackrel{\fa\in\NN^n,\,\fb\in\NN^m}{|\fa|+|\fb|\leq N}} \|D^{\fa}_xD^{\fb}_yu(T)\varphi\|^2_{L^2(\cB_R;d\mu_{\fa})} + I_N
&\leq C\sum_{\stackrel{\fa\in\NN^n,\,\fb\in\NN^m}{|\fa|+|\fb|\leq N}}\|D^{\fa}_xD^{\fb}_yu(t)\varphi\|^2_{L^2(\cB_R;d\mu_{\fa})} 
+ C I_{N-1}  \\
&\quad + \eps I_N+ \eps \left(\|u(0)\|^2_{L^2(\cB_R;d\mu)} + \|u\|^2_{L^2((0,T);L^2(\cB_R;d\mu))}\right).
\end{align*} 
Choosing $\eps=1/2$ and applying the induction hypothesis \eqref{eq:Induction_I_N} to $I_{N-1}$, we obtain
\begin{align*}
\sum_{\stackrel{\fa\in\NN^n,\,\fb\in\NN^m}{|\fa|+|\fb|\leq N}} \|D^{\fa}_xD^{\fb}_yu(T)\varphi\|^2_{L^2(\cB_R;d\mu_{\fa})} + I_N
&\leq C\sum_{\stackrel{\fa\in\NN^n,\,\fb\in\NN^m}{|\fa|+|\fb|\leq N}}\|D^{\fa}_xD^{\fb}_yu(t)\varphi\|^2_{L^2(\cB_R;d\mu_{\fa})} \\
&\quad + C \left(\|u(0)\|^2_{L^2(\cB_R;d\mu)} + \|u\|^2_{L^2((0,T);L^2(\cB_R;d\mu))}\right).
\end{align*} 
For $T_0\in (0,T)$, we integrate the preceding inequality in the $t$-variable on the interval $[T_0,T]$ and we obtain
\begin{align*}
&(T-T_0)\sum_{\stackrel{\fa\in\NN^n,\,\fb\in\NN^m}{|\fa|+|\fb|\leq N}}\|D^{\fa}_xD^{\fb}_yu(T)\varphi\|^2_{L^2(\cB_R;d\mu_{\fa})} 
+ (T-T_0)I_N\\
&\quad\leq C\sum_{\stackrel{\fa\in\NN^n,\,\fb\in\NN^m}{|\fa|+|\fb|\leq N}}
\int_{T_0}^T\|D^{\fa}_xD^{\fb}_yu(t)\varphi\|^2_{L^2(\cB_R;d\mu_{\fa})} \,dt\\
&\quad\quad+C(T-T_0) \left(\|u(0)\|^2_{L^2(\cB_R;d\mu)} + \|u\|^2_{L^2((0,T);L^2(\cB_R;d\mu))}\right).
\end{align*}
Applying again the induction hypothesis \eqref{eq:Induction_I_N}  to the first term on the right-hand side of the preceding inequality, we obtain estimates \eqref{eq:Higher_order_derivatives}, for all $(\fa,\fb)\in\NN^n\times\NN^m$ such that $|\fa|+|\fb|=N$, and \eqref{eq:Induction_I_N}. This completes the proof.
\end{proof}

We have the following approximation result with smooth functions.

\begin{lem}[Approximation with smooth function]
\label{lem:Approx_smooth}
Assume that the operator $L$ satisfies condition \eqref{eq:Operator_adapted_system} and Assumption \ref{assump:Coeff}. Let $0<r<R<1$, $T>0$, $f\in L^2(\cB_R; d\mu)$, and let $u$ be a local weak solution to the initial-value problem \eqref{eq:Initial_value_problem_local}. Then there are sequences of smooth functions, 
$\{f_k\}_{k\in\NN}\subset C^{\infty}_c(\bar\cB_r\backslash\partial^T \cB_r)$ and $\{u_k\}_{k\in\NN}\subset C^{\infty}_c([0,T]\times\bar\cB_r)$, that satisfy the following properties. For all $k\in\NN$, the function $u_k$ is a local weak solution to
\begin{equation}
\label{eq:Problem_k}
\begin{aligned}
\left\{\begin{array}{rl}
\partial_t u_k-L u_k= 0 & \hbox{ on } (0, T)\times \cB_r,\\ 
 u_k(0) = f_k& \hbox{ on } \cB_r, 
\end{array} \right.
\end{aligned}
\end{equation}
and has the property that
\begin{align}
\label{eq:u_k_boundary}
&u_k = 0\quad\hbox{ on } (0,T)\times (\partial\cB_r\cap\{x_i=0\}),\quad\forall\, 1\leq i\leq n_0,\\
\label{eq:Space_for_derivatives}
&D^{\fa}_xD^{\fb}_yu_k\in L^2((0,T); H^1(\cB_r;d\mu_{\fa})),\quad\forall\,k\in\NN,\quad\forall\,(\fa,\fb)\in\NN^n\times\NN^m.
\end{align}
Moreover, there is a positive constant $C$ such that
\begin{equation}
\label{eq:u_k_boundedness}
\sup_{k\in\NN}\|u_k\|_{L^2((0,T);L^2(\cB_r;d\mu))} + \sup_{k\in\NN}\|f_k\|_{L^2(\cB_r;d\mu)} \leq C,
\end{equation}
and, as $k$ tends to $\infty$, we have that 
\begin{align}
\label{eq:f_k_conv}
f_k \rightarrow f &\quad\hbox{ strongly in } L^2(\cB_r;d\mu),\\
\label{eq:u_k_conv}
u_k \rightarrow u &\quad\hbox{ strongly in } L^2((0,T); H^1(\cB_r;d\mu)),\\
\label{eq:u_k_time_conv}
\frac{du_k}{dt} \rightarrow \frac{du}{dt}&\quad\hbox{ strongly in } L^2((0,T); H^{-1}(\cB_r;d\mu)),\\
\label{eq:u_k_derivatives_conv}
D^{\fa}_xD^{\fb}_y u_k \rightarrow D^{\fa}_xD^{\fb}_yu &\quad\hbox{ pointwise on } (0,T)\times\cB_r,
\end{align}
for all $\fa\in\NN^n$ and $\fb\in\NN^m$.
\end{lem}

\begin{proof}
Let $P := [0,2R]^n\times[-2R,2R]^m$ and let $\cL$ be an operator defined on $P$ such that $\cL$ satisfies Assumption \ref{assump:Coeff} and $\cL v = L v$, for all function $v \in C^{\infty}(\bar \cB_R)$. Let $r<r_0<R$ and $\varphi:P \rightarrow[0,1]$ be a smooth function such that $\varphi\restrictedto_{\bar\cB_{r_0}} = 1$ and $\varphi\restrictedto_{\bar\cB_R^c} = 0$. Let 
$$
\bar u:= \varphi u,\quad\bar f:=\varphi f,\quad\bar g:=[L,\varphi]u.
$$
It follows from Lemma \ref{eq:Cutoff_local_solution} that $\bar u$ is a weak solution to equation \eqref{eq:Problem_cutoff} where we replace $\cB_R$ by $P$, i.e.
\begin{equation}
\label{eq:Problem_cutoff_u}
\begin{aligned}
\left\{\begin{array}{rl}
\bar u_t-\cL\bar u=\bar g & \hbox{ on } (0, T)\times P,\\ 
\bar u(0) = \bar f& \hbox{ on } P, 
\end{array} \right.
\end{aligned}
\end{equation}
Because $\varphi = 1$ on $\bar\cB_{r_0}$, it follows from \eqref{eq:Commutator} that $\bar g = 0$ on $\bar \cB_{r_0}$, and so using the fact that the space of functions $C^{\infty}_c([0,T]\times (P\backslash\partial^T P))$ is dense in $L^2((0,T);L^2(P;d\mu))$, we can find a sequence of functions such that $\{g_k\}_{k\in\NN}\subset C^{\infty}_c([0,T]\times (P\backslash\partial^T P))$, $g_k = 0$ on $\bar \cB_{r_0}$, and
\begin{equation}
\label{eq:g_k_conv}
g_k \rightarrow \bar g \quad\hbox{ strongly in } L^2((0,T); L^2(P;d\mu)),\quad\hbox{ as $k\rightarrow\infty$}.
\end{equation}
We recall the definition of the portion of the boundary $\partial^T P$ in \eqref{eq:Tangent_boundary}. Similarly, we can find a sequence of functions, $\{f_k\}_{k\in\NN}\subset C^{\infty}_c(P\backslash\partial^T P)$, such that
\begin{equation}
\label{eq:f_k_conv_bar_f}
f_k \rightarrow \bar f \quad\hbox{ strongly in } L^2(P;d\mu),\quad\hbox{ as $k\rightarrow\infty$},
\end{equation}
which gives \eqref{eq:f_k_conv}, since $\bar f=f$ on $\cB_r$.
Applying \cite[Lemma 10.0.2]{Epstein_Mazzeo_annmathstudies}, the parabolic problem
\begin{equation}
\label{eq:Problem_k_P}
\begin{aligned}
\left\{\begin{array}{rl}
\partial_t u_k-\cL u_k = g_k & \hbox{ on } (0, T)\times P,\\ 
u_k(0) = f_k& \hbox{ on } P, 
\end{array} \right.
\end{aligned}
\end{equation}
has a unique solution $u_k\in C^{\infty}([0,T]\times P)$. Let $1\leq i\leq n_0$, where $n_0$ is defined in \eqref{eq:n_0}, and let 
$P_{i_0}:=P\cap\{x_{i_0}=0\}$. For all smooth functions $v\in C^{\infty}(P)$, it follows by an observation of Sato, 
\cite{Sato_1978}, \cite[Lemma 2.4]{Shimakura_1981}, that $(\cL v)\restrictedto_{P_{i_0}} = \cL\restrictedto_{P_{i_0}} v$, 
where $\cL\restrictedto_{P_{i_0}}$ is an operator that satisfies Assumption \ref{assump:Coeff}. Using the fact that $g_k=0$ on $[0,T]\times P_{i_0}$ and $f_k=0$ on $P_{i_0}$, it follows that
$u_k$ is a smooth solution, when restricted to $[0,T]\times P_{i_0}$, to the parabolic problem
$$
(\partial_t-\cL\restrictedto_{P_{i_0}}) u_k = 0\hbox{ on } (0,T)\times P_{i_0}
\quad\hbox{ and }\quad
u_k(0) = 0\hbox{ on } P_{i_0}.
$$
Applying the uniqueness statement in \cite[Theorem 10.0.2]{Epstein_Mazzeo_annmathstudies}, it follows that $u_k=0$ on $[0,T]\times P_{i_0}$, which implies that \eqref{eq:u_k_boundary} holds. We also have that $D^{(\fa,\fb)} u_k = 0$ on $P_{i_0}$, for all $(\fa,\fb)\in\NN^n\times\NN^m$ such that $\fa_{i_0} = 0$. From the definition of the weight function $d\mu_{\fa}$ in \eqref{eq:Weight_for_higher_order_Sobolev_spaces}, it follows that \eqref{eq:Space_for_derivatives} holds, and so also that $u_k$ belongs to the space of functions $\cF((0,T)\times P)$. Thus, each function $u_k$ is also a weak solution to the local problem \eqref{eq:Problem_k}, where we use the fact that $g_k=0$ on $\bar \cB_r$. The global Sobolev estimates \eqref{eq:Energy_estimate} applied to the solutions $u_k$ and $\bar u$, combined with properties \eqref{eq:g_k_conv} and \eqref{eq:f_k_conv} imply properties \eqref{eq:u_k_conv} and \eqref{eq:u_k_time_conv}, while together with \eqref{eq:f_k_conv}, we obtain \eqref{eq:u_k_boundedness}. Because the operator $L$ is strictly elliptic on $\cB_r$, the standard elliptic estimates and property \eqref{eq:u_k_conv} give us that the pointwise convergence \eqref{eq:u_k_derivatives_conv} holds. This completes the proof.
\end{proof}

We conclude this section with the

\begin{proof}[Proof of Theorem \ref{thm:Regularity_Sobolev}]
Let $r<r_0<R$. Let $\{u_k\}_{k\in\NN}\in C^{\infty}([0,\infty)\times \bar\cB_{r_0})$ be the sequence of functions constructed in Lemma \ref{lem:Approx_smooth}, when we apply it with $r$ replaced by $r_0$. Properties \eqref{eq:Space_for_derivatives}, \eqref{eq:f_k_conv}, \eqref{eq:u_k_conv}, and estimate 
\eqref{eq:Higher_order_derivatives} applied to $u_k$ with $R=r_0$, gives us that, for all $\fa\in\NN^n$ and $\fb\in\NN^m$,
$\{D^{\fa}_xD^{\fb}_yu_k\}_{k\in\NN}$ is a Cauchy sequence in $L^{\infty}((0,T);L^2(\cB_r;d\mu_{\fa}))$ and there is a positive constant, $C=C(L,t,T,r,r_0)$, such that
\begin{equation*}
\sup_{s\in[t,T]}
\|D^{\fa}_xD^{\fb}_yu_k(s)\|^2_{L^2(\cB_r;d\mu_{\fa})}
\leq C\left(\|f_k\|^2_{L^2(\cB_{r_0};d\mu)}+\|u_k\|^2_{L^2((0,T);L^2(\cB_{r_0};d\mu))}\right),\quad\forall\,k\in\NN.
\end{equation*}
The preceding observations combined with properties \eqref{eq:u_k_derivatives_conv} and \eqref{eq:u_k_boundedness} yield  \eqref{eq:Regularity_Sobolev}. This completes the proof.
\end{proof}

\section{Boundary behavior of weak solutions}
\label{sec:Supremum_estimates}
In this section we give the proofs of the main results stated in \S\ref{sec:Local_sup_est} and \S\ref{sec:Global_regularity}. 

\subsection{Boundary behavior of local weak solutions}
\label{sec:Sup_est_proof}
We begin with the proof of Theorem \ref{thm:Boundary_reg}, which is a consequence of a weaker form of the supremum estimates established in Theorem \ref{thm:Sup_est} below. To state the result, we need to introduce some additional notation. For all $z\in \bar S_{n,m}$, $\rho\in\bar\RR_+^{n_0}$, 
and $r>0$, we set 
\begin{align}
\label{eq:Projection_pi}
\pi(z,\rho) &:= (\rho,x_{n_0+1},\ldots,x_n, y)\in\RR^{n+m},\\
\label{eq:Definition_cR}
\cR_{r_0}(z)&:=\prod_{i=1}^n 
(x_i, r_0)\times\prod_{l=1}^m (y_l,r_0),\quad\forall\, z\in\cB_r.
\end{align}
When $\rho=0\in\RR^{n_0}$, we write for brevity $\pi(z)$ instead of $\pi(z,0)$. For all $z\in \bar S_{n,m}$, 
$\eta\in\bar\RR_+^{n_0-1}$, $1\leq k\leq n_0$, and $r>0$, we set 
\begin{equation}
\label{eq:Projection_pi_k}
\pi_k(z,\eta) := (\eta_1,\ldots,\eta_{k-1},x_k,\eta_k,\ldots, \eta_{n_0-1},x_{n_0+1},\ldots,x_n, y)\in\RR^{n+m},
\end{equation}
When $\eta=0\in\RR^{n_0-1}$, we write for brevity $\pi_k(z)$ instead of $\pi_k(z,0)$. We can now state

\begin{thm}[Local a priori supremum estimates of solutions]
\label{thm:Sup_est}
Assume that the operator $L$ in \eqref{eq:Operator} satisfies \eqref{eq:Operator_adapted_system} and Assumption \ref{assump:Coeff}. Let $T>0$, $R\in (0,1)$, $f\in L^2(\cB_R; d\mu)$, and let $u$ be a local weak solution to the initial-value problem \eqref{eq:Initial_value_problem_local}. Then for all $0<r<r_0<R$, $0<t<T$, and $\fb\in\NN^m$, there are positive constants,
$C=C(\fb,L,r,r_0,R,t,T)$, $p=p(b,m,n)>2$, and $q=q(b,m,n)<2$, with the property that $1/p+1/q=1$, such that for all $s\in [t,T]$ and for all $z\in \cB_r$, we have that
\begin{align}
\label{eq:Sup_est_sol_1}
|D^{\fb}_y u(s,z)| &\leq C 
\left(\|f\|_{L^2(\cB_R;d\mu)} + \|u\|_{L^2((0,T);L^2(\cB_R;d\mu))}\right) \cW_{r_0}(\pi(z)) \prod_{i=1}^{n_0} x_i,
\end{align}
and for all $1\leq k\leq n_0$, we have that
\begin{align}
\label{eq:Sup_est_sol_2}
|D^{\fe_k}_xD^{\fb}_y u(s,z)| 
&\leq C \left(\|f\|_{L^2(\cB_R;d\mu)} + \|u\|_{L^2((0,T);L^2(\cB_R;d\mu))}\right) \cW_{r_0}(\pi_k(z)) 
\prod_{\stackrel{i=1}{i \neq k}}^{n_0} x_i,
\end{align}
and for all $n_0+1\leq l\leq n$, we have that
\begin{align}
\label{eq:Sup_est_sol_3}
|D^{\fe_l}_xD^{\fb}_yu(s,z)| 
&\leq C \left(\|f\|_{L^2(\cB_R;d\mu)} + \|u\|_{L^2((0,T);L^2(\cB_R;d\mu))}\right) \cW^l_{r_0}(\pi(z)) \prod_{i=1}^{n_0} x_i,
\end{align}
where we define
\begin{align}
\label{eq:Definition_cW}
\cW_{r_0}(z) &:= \left(\int_{\cR_{r_0}(z)}
\prod_{i=1}^{n_0} \xi_i^{- \frac{q}{p}} \, d\xi_i
\prod_{j=n_0+1}^{n}  \xi_j^{-b_j(\xi,\rho) \frac{q}{p}}\, d\xi_j
\prod_{l=1}^md\rho_l
\right)^{\frac{1}{q}},\\
\label{eq:Definition_cW_k}
\cW^l_{r_0}(z) &:=
\left(\int_{\cR_{r_0}(z)}
\prod_{i=1}^{n_0}  \xi_i^{- \frac{q}{p}} \, d\xi_i
\prod_{j=n_0+1}^{n} \xi_j^{-( b_j(\xi,\rho)+\delta_{jl}) \frac{q}{p}}\, d\xi_j
\prod_{l=1}^md\rho_l
\right)^{\frac{1}{q}}.
\end{align}
\end{thm}

We first give the proof of Theorem \ref{thm:Sup_est} under a stronger hypothesis on the regularity of the weak solution; see Theorem \ref{thm:Sup_est_smooth}. We then employ the approximation property described in Lemma \ref{lem:Approx_smooth} to remove this unnecessary strong hypothesis. Let $\fe\in\NN^n$ be defined by
\begin{equation}
\label{eq:fe}
\fe := \fe_1+\fe_2 +\ldots+\fe_{n_0}.
\end{equation}
where the integer $n_0$ is defined in \eqref{eq:n_0}. We have

\begin{thm}[Local supremum estimates for smooth solutions]
\label{thm:Sup_est_smooth}
Assume that the operator $L$ satisfies condition \eqref{eq:Operator_adapted_system} and Assumption \ref{assump:Coeff}. Let $T>0$, $R\in (0,1)$, $f\in L^2(\cB_R; d\mu)$, and $u\in C^{\infty}([0,T]\times\bar\cB_R)$ be a weak solution to the initial-value problem \eqref{eq:Initial_value_problem_local} such that
\begin{equation}
\label{eq:Boundary_values_tangent}
u = 0\quad\hbox{ on } (0,T)\times(\partial\cB_R\cap\{x_i=0\}),\quad\forall\, 1\leq i\leq n_0,
\end{equation}
where the integer $n_0$ is defined in \eqref{eq:n_0}. Then for all $0<r<R$, $0<t<T$, and $\fb\in\NN^m$, there are positive constants
$C=C(\fb,L,r,R,t,T)$, $p=p(b,m,n)>2$, and $q=q(b,m,n)<2$, such that for all $s\in [t,T]$ and for all 
$z\in \bar\cB_r$, estimates \eqref{eq:Sup_est_sol_1}, \eqref{eq:Sup_est_sol_2}, and \eqref{eq:Sup_est_sol_3} hold.
\end{thm}

\begin{proof}
We will prove estimate \eqref{eq:Sup_est_sol_2} only for $k=1$, because the argument is identical for all $1\leq k\leq n_0$. Similarly, we will prove estimate \eqref{eq:Sup_est_sol_3} only for $k=n$, because the argument is identical for all $n_0+1\leq k\leq n$.
Notice that property \eqref{eq:Boundary_values_tangent} implies that for all $1\leq i\leq n_0$, $\fa\in \NN^n$ such that $\fa_i=0$, and $\fb\in\NN^m$, we have that $D^{\fa}_xD^{\fb}_y u = 0$ on $(0,T)\times(\partial\cB_R\cap\{x_i=0\})$. This implies that for all $(s,z) \in (0,T)\times\cB_r$ we have
\begin{align}
\label{eq:Sup_est_sol_u_N_1}
D^{\fb}_yu(s,z) &= \int_0^{x_1}\ldots\int_0^{x_{n_0}} D^{\fe}_xD^{\fb}_yu(s,\pi(z,\rho))\,d\rho,\\
\label{eq:Sup_est_sol_u_N_2}
D^{\fe_1}_xD^{\fb}_yu_l(s,z) &= \int_0^{x_2}\ldots\int_0^{x_{n_0}} D^{\fe}_xD^{\fb}_yu(s,\pi_1(z,\eta))\,d\eta,\\
\label{eq:Sup_est_sol_u_N_3}
D^{\fe_n}_xD^{\fb}_yu_l(s,z) &= \int_0^{x_1}\ldots\int_0^{x_{n_0}} D^{\fe_n+\fe}_xD^{\fb}_yu(s,\pi(z,\rho))\,d\rho.
\end{align}
where we we recall the definition of $\pi(z,\rho)$ in \eqref{eq:Projection_pi} and of $\pi_1(z,\eta)$ in \eqref{eq:Projection_pi_k} 
($k=1$). Inequality \eqref{eq:Sup_est_deriv_slice} applied in identity \eqref{eq:Sup_est_sol_u_N_1} gives us
\begin{align}
|D^{\fb}_yu(s,z)| 
&\leq C\cN' \int_0^{x_1}\ldots\int_0^{x_{n_0}} \cW_{r_0}(\pi(z,\rho)) \,d\rho,\notag\\
&\leq C\cN' \int_0^{x_1}\ldots\int_0^{x_{n_0}} \cW_{r_0}(\pi(z)) \,d\rho,\notag\\
\label{eq:Bound_D_0_b}
&\leq C\cN' \cW_{r_0}(\pi(z)) \prod_{i=1}^{n_0} x_i,
\end{align}
where $\cW$ is defined in \eqref{eq:Definition_cW} and $\cN'$ is given by
$$
\cN':=\sum_{(\fa,\fc)\in\cD} \|D^{\fa+\fe}_xD^{\fc+\fb}_yu(s)\|_{H^1(\cB_r;d\mu_{\fa+\fe})},
$$
with $\cD$ defined in~\eqref{eq:Definition_cD}.  Using estimate
\eqref{eq:Regularity_Sobolev} to bound $\cN'$ and applying inequality
\eqref{eq:Bound_D_0_b}, we obtain \eqref{eq:Sup_est_sol_1}. The proof of
estimate \eqref{eq:Sup_est_sol_1} is identical to that of
\eqref{eq:Sup_est_sol_1} with the only modification that we integrate in
\eqref{eq:Sup_est_sol_u_N_2} in $n_0-1$ variables and we replace
$\cW_{r_0}(\pi(z))$ by $\cW_{r_0}(\pi_1(z))$. Lastly, the proof of inequality
\eqref{eq:Sup_est_sol_u_N_3} is also identical to that of
\eqref{eq:Sup_est_sol_u_N_1} with the only modification that we apply the
argument with $u$ replaced by $D^{\fe_n}_xu$ and with the measure $d\mu$
defined in \eqref{eq:Weight} replaced with the measure $d\mu_{\fe_n}$ defined
in \eqref{eq:Weight_for_higher_order_Sobolev_spaces}, with $\fa:=\fe_n$.  This
concludes the proof.
\end{proof}

We can now give the 

\begin{proof}[Proof of Theorem \ref{thm:Sup_est}]
Let $\{u_k\}_{k\in\NN}$ be the sequence constructed in 
Lemma \ref{lem:Approx_smooth}, which we apply with $r=(r_0+R)/2$. Properties \eqref{eq:Problem_k} and \eqref{eq:u_k_boundary} and Theorem \ref{thm:Sup_est_smooth} applied to $u_k$ implies that estimates \eqref{eq:Sup_est_sol_1}, \eqref{eq:Sup_est_sol_2}, and \eqref{eq:Sup_est_sol_3} hold with $u$ replaced by $u_k$. Combining this with property \eqref{eq:u_k_derivatives_conv} yields that estimates \eqref{eq:Sup_est_sol_1}, \eqref{eq:Sup_est_sol_2}, and \eqref{eq:Sup_est_sol_3} hold for $u$.
\end{proof}

We next combine the pointwise estimates derived in Theorem \ref{thm:Sup_est} with a conjugation property of generalized Kimura operators, which is a small variation of a procedure known in probability as
Doob's $h$-transform, to give the proof of Theorem \ref{thm:Boundary_reg}. Let 
\begin{equation}
\label{eq:Weight_density_transverse}
w^{\pitchfork}(z) := \prod_{j=n_0+1}^{n} x_j^{b_j(z)-1},
\quad\forall\, z\in S_{n,m}.
\end{equation}
Then we can write $d\mu(z) = w^T(z) w^{\pitchfork}(z)\, dz$, where we recall the definition of $w^T(z)$ in \eqref{eq:Weight_density_tangent}. Given a function $u$ we denote
\begin{equation}
\label{eq:tilde_u}
\widetilde u := w^T u,
\end{equation}
and direct calculations give us that
\begin{equation}
\label{eq:Conjugation_op}
\widetilde L \widetilde u = w^T L ((w^T)^{-1} \widetilde u),
\end{equation}
where $\widetilde L$ is a generalized Kimura operator with \emph{positive} weights defined by
\begin{equation}
\label{eq:tilde_L}
\begin{aligned}
\widetilde L &:= L 
+ \sum_{i=1}^{n_0}\left(2+2x_i\sum_{j=1}^{n_0} a_{ij}(z)\right)\partial_{x_i}
+ \sum_{j=n_0+1}^n 2x_j\sum_{i=1}^{n_0} a_{ij}(z) \partial_{x_j}\\
&\quad
+ \sum_{l=1}^m\sum_{i=1}^{n_0} c_{il}(z)\partial_{y_l}
+\left(\sum_{i,j=1}^{n_0} a_{ij}(z)+\sum_{i=1}^{n_0}\frac{b_i(z)}{x_i}\right).
\end{aligned}
\end{equation}
Similarly to the weight $d\mu$ in \eqref{eq:Weight}, we associate to the operator $\widetilde L$ the weight
\begin{equation}
\label{eq:tilde_weight}
d\widetilde\mu(z):= (w^T(z))^{-1}w^{\pitchfork}(z)\, dz.
\end{equation}
We expect that if $u$ is a local weak solution to the equation $u_t-Lu=0$ on $(0,T)\times\cB_R$, then $\widetilde u$ is a local weak solution to $\widetilde u_t-\widetilde L\widetilde u=0$ on $(0,T)\times\cB_R$. However, this is not obvious because, given a function $v\in H^1(\cB_R;d\mu)$, it does not follow in general that $\widetilde v$ belongs to $H^1(\cB_R;d\widetilde\mu)$. To see this, we consider the first-order derivative $\widetilde v_{x_i}$, for all $1\leq i\leq n_0$, which are given by
\begin{equation}
\label{eq:Derivative_tilde_v}
\widetilde v_{x_i} = -\frac{w^T}{x_i} v+w^T v_{x_i}.
\end{equation}
We notice that the second term on the right-hand side satisfies
$$
\int_{\cB_R} x_{i}|w^Tv_{x_i}|^2\, d\widetilde\mu = \int_{\cB_R} x_{i}|v_{x_i}|^2\, d\mu, 
$$
which is finite because $v\in H^1(\cB_R;d\mu)$. Thus, to conclude that 
$
\int_{\cB_R} x_{i}|\widetilde v_{x_i}|^2\, d\widetilde\mu <\infty,
$
we need that
$$
\int_{\cB_R} x_{i}\left|\frac{w^T}{x_i} v\right|^2\, d\widetilde\mu = \int_{\cB_R} \frac{1}{x_i}|v|^2\, d\mu, 
$$
is finite, which is not true in general.

In Lemma \ref{eq:tilde_u_space} below we prove with the aid of the
supremum bounds established in Theorem \ref{thm:Sup_est}, that knowing in
addition that $u$ is a local weak solution to the equation $u_t-Lu=0$ on
$(0,T)\times\cB_R$, then the function $\widetilde u$ belongs to the functional
space $L^2((t,T);H^1(\cB_r;d\widetilde\mu))$, for all $t\in(0,T)$ and for all
$r<R$, and so we can proceed to establish in Lemma \ref{eq:tilde_u_solution}
that $\widetilde u$ is a local weak solution to the equation $\widetilde
u_t-\widetilde L\widetilde u=0$ on $(0,T)\times\cB_R$.

\begin{lem}
\label{eq:tilde_u_space}
Let $u$ be a local weak solution to equation \eqref{eq:Problem_local}. Then $\widetilde u$ defined in \eqref{eq:tilde_u} belongs to 
$L^2((t,T);H^1(\cB_r;d\widetilde\mu))$, for all $t\in (0,T)$ and for all $r<R$.
\end{lem}

\begin{proof}
From identity \eqref{eq:tilde_u}, we have that
$$
\widetilde u_{x_j} = w^Tu_{x_j},\quad\widetilde u_{y_l}=u_{y_l},\quad\forall\, n_0+1\leq j\leq n,\quad\forall\, 1\leq l\leq m,
$$
and using \eqref{eq:tilde_weight}, it is clear that 
$$\|u(s)\|_{L^2(\cB_R;d\mu)}=\|\widetilde u(s)\|_{L^2(\cB_R;d\widetilde\mu)}, 
\|w^T u_{x_i}(s)\|_{L^2(\cB_R; x_id\widetilde\mu)}=\|u_{x_i}(s)\|_{L^2(\cB_R;
  x_id\mu)},$$ 
for all $1\leq i\leq n$, and that $\|w^T u_{y_l}(s)\|_{L^2(\cB_R;
  x_id\widetilde\mu)}=\|u_{y_l}(s)\|_{L^2(\cB_R; x_id\mu)}$, for all $1\leq
l\leq m$, where $s\in (0,T)$. Thus, using \eqref{eq:Derivative_tilde_v} applied
with $v=u(s)$, it remains to show that
\begin{equation}
\label{eq:Deriv_u}
\int_t^T\left\|\frac{w^T}{x_i} u(s)\right\|^2_{L^2(\cB_r; x_id\widetilde\mu)}\,ds<\infty,
\end{equation}
for all $0<t<T$ and $0<r<R$. Let $r<r_0<R$. Estimate \eqref{eq:Sup_est_sol_1} gives us that there is a positive constant, $C=C(\fb,L,t,T,r,r_0,R)$, such that
\begin{equation}
\label{eq:Bound_u}
|u(s,z)| \leq C \cW_{r_0}(\pi(z)) \prod_{i=1}^{n_0} x_i,\quad\forall\, s\in (t,T),\quad\forall\, z\in\cB_r.
\end{equation}
Let $\delta>0$ be small enough so that
\begin{equation}
\label{eq:Cond_b_j_1}
b_j(O)\left(1-\frac{2}{p}\right)+\frac{2}{q}-\delta\left(1+\frac{2}{p}\right)>0,
\quad\forall\, n_0+1\leq j\leq n,
\end{equation} 
where $p>2$ and $q<2$ are the constants appearing in the statement of Theorem \ref{thm:Sup_est}. Without loss of generality we can assume that $R$ is small enough so that 
\begin{equation}
\label{eq:Cond_b_j_2}
b_j(O)-\delta <b_j(z)\leq b_j(O)+\delta,\quad\forall\, z\in \cB_R,\quad\forall\, n_0+1\leq j\leq n.
\end{equation} 
Inequalities \eqref{eq:Bound_u}, \eqref{eq:Cond_b_j_2}, together with \eqref{eq:Definition_cW}, \eqref{eq:Definition_cR}, and \eqref{eq:Projection_pi} give us that
\begin{align*}
&\int_t^T\left\|\frac{w^T}{x_i} u(s)\right\|^2_{L^2(\cB_R; x_id\widetilde\mu)}\,ds\\
&\qquad\leq 
C(T-t)\int_{\cB_r} \prod_{j=n_0+1}^n 1\vee x_j^{-(b_j(O)+\delta)\frac{2}{p}+\frac{2}{q}} \prod_{j=n_0+1}^n x_j^{b_j(O)-\delta-1}x_i^{-1}(w^T(z))^{-1}\,dz,
\end{align*}
which gives us that \eqref{eq:Deriv_u} holds using condition \eqref{eq:Cond_b_j_1}. This concludes the proof.
\end{proof}

We can now give the proof of

\begin{lem}
\label{eq:tilde_u_solution}
Let $u$ be a local weak solution to equation \eqref{eq:Problem_local}. Then $\widetilde u$ defined in \eqref{eq:tilde_u} is a local weak solution to equation
\begin{equation}
\label{eq:Problem_local_tilde_u}
\widetilde u_t-\widetilde L\widetilde u = 0 \hbox{ on } (0,T)\times\cB_R.
\end{equation}
\end{lem}

\begin{proof}
Let $\widetilde Q(\cdot,\cdot)$ be the bilinear form associated to the operator $\widetilde L$ as in \S\ref{sec:Weak_sol}.
It is sufficient to prove that for all test functions 
$v\in C^{\infty}_c((0,T]\times\underline\cB_R)$, we have that
\begin{equation}
\label{eq:Weak_sol_var_eq_tilde}
(\widetilde u(T), v(T))_{L^2(\Omega;d\widetilde\mu)} - \int_0^T\left\langle \widetilde u(t), \frac{dv(t)}{dt}\right\rangle\, dt
+\int_0^T \widetilde Q(\widetilde u(t),v(t))\, dt= 0,
\end{equation}
where $\left\langle\cdot, \cdot\right\rangle$ denotes the dual pairing of $H^1(\cB_R;d\widetilde\mu)$ and $H^{-1}(\cB_R;d\widetilde\mu)$. Combining identity
$$
\int_0^T\left\langle \frac{\widetilde u(t)}{dt}, dv(t)\right\rangle\, dt
= (\widetilde u(T), v(T))_{L^2(\Omega;d\widetilde\mu)} - \int_0^T\left\langle \widetilde u(t), \frac{dv(t)}{dt}\right\rangle\, dt 
$$
understood in the sense of distributions, together with \eqref{eq:Weak_sol_var_eq_tilde} and Lemma \ref{eq:tilde_u_space}, it follows by Definition \ref{defn:Weak_sol_local} (b) that $\widetilde u$ is indeed a weak solution to equation \eqref{eq:Problem_local_tilde_u}. 

Let $\eps>0$ and denote $\cB^{\eps}_R:=\cB_R\cap\{x_i>\eps:1\leq i\leq n\}$. Because the operator $L$ is strictly elliptic on $\cB^{\eps}_R$, we have that $\widetilde u_t-\widetilde L\widetilde u=0$ is satisfied on $(0,T)\times\cB^{\eps}_R$ in the classical sense, and so
\begin{equation}
\label{eq:Weak_sol_var_eq_tilde_eps}
(\widetilde u(T)\mathbf{1}_{\cB^{\eps}_R}, v(T))_{L^2(\Omega;d\widetilde\mu)} -\int_0^T\left\langle \widetilde u(t)\mathbf{1}_{\cB^{\eps}_R}, \frac{dv(t)}{dt}\right\rangle\, dt
+\int_0^T \widetilde Q_{\eps}(\widetilde u(t),v(t))\, dt
+\sum_{i=1}^n  I^{\eps}_i= 0,
\end{equation}
where $\widetilde Q_{\eps}(\cdot,\cdot)$ is the bilinear form associated to the generalized Kimura operator $\widetilde L$ restricted to the domain $\cB^{\eps}_R$, obtained analogously to \S\ref{sec:Weak_sol}, and the boundary terms $I^{\eps}_i$ are defined by:
\begin{equation}
\label{eq:I_eps}
\begin{aligned}
I^{\eps}_i
&:=\int_0^T\int_{\{x_i=\eps\}}\left(x_i\widetilde u_{x_i}(t)+\sum_{j=1}^nx_ix_ja_{ij}\widetilde u_{x_j}(t)+\sum_{l=1}^mx_ic_{il}\widetilde u_{y_l}(t)\right)v(t)\,d\widetilde\mu\, dt,
\end{aligned}
\end{equation}
for all $1\leq i\leq n$. Identity \eqref{eq:Derivative_tilde_v} applied with $v=u(t)$, and definition \eqref{eq:tilde_weight} of the weight $d\widetilde\mu$ give us that
\begin{align*}
I^{\eps}_i
&:=\int_0^T\int_{\{x_i=\eps\}}\left(x_i u_{x_i}(t)+\sum_{j=1}^nx_ix_ja_{ij} u_{x_j}(t)+\sum_{l=1}^mx_ic_{il} u_{y_l}(t)\right)v(t) w^{\pitchfork}\,dz\, dt\\
&\quad+ \int_0^T\int_{\{x_i=\eps\}}\left(-\mathbf{1}_{\{i\leq n_0\}}-\sum_{j=1}^{n_0}x_ja_{ij}+\sum_{l=1}^mx_ix_{il} \right)u(t)v(t) w^{\pitchfork}\,dz\, dt.
\end{align*}
Let $\delta>0$ be small enough so that
\begin{equation}
\label{eq:Cond_b_j_3}
b_j(O)\left(1-\frac{1}{p}\right)+\frac{1}{q}-\frac{1}{p}-\delta\left(1+\frac{1}{p}\right)>0,\quad\forall\, n_0+1\leq j\leq n,
\end{equation} 
where $p>2$ and $q<2$ are the constants appearing in the statement of Theorem \ref{thm:Sup_est}. Without loss of generality we can assume that $R$ is small enough so that condition \eqref{eq:Cond_b_j_2} holds. Let $0<t<T$ and $0<r<R$ be such that $\hbox{supp}(v)\subseteq [t,T]\times\bar\cB_r$. Then the supremum estimates in Theorem \ref{thm:Sup_est} and \eqref{eq:I_eps} give us that there is a positive constant, $C=C(\fb,L,t,T,r,R)$, such that 
$$
|I^{\eps}_i|\leq C\int_{\{x_i=\eps\}\cap\cB_r}\prod_{k=1}^{n_0} x_k
\left(\cW_{r_0}(\pi(z))+\sum_{j=1}^{n_0}\cW_{r_0}(\pi_j(z))+\sum_{l=n_0+1}^n \cW^l_{r_0}(\pi(z))\right)
w^{\pitchfork}(z)\,dz,
$$
where $r<r_0<R$. From definitions \eqref{eq:Definition_cW}, \eqref{eq:Definition_cW_k}, \eqref{eq:Definition_cR}, \eqref{eq:Projection_pi}, \eqref{eq:Projection_pi_k} and condition \eqref{eq:Cond_b_j_2}, it follows that
$$
|I^{\eps}_i|\leq C\int_{\{x_i=\eps\}\cap\cB_r}\prod_{k=1}^{n_0} x_k
\prod_{j=n_0+1}^n\left(1\vee x_j^{-(b_j(O)+\delta+1)\frac{1}{p}+\frac{1}{q}}\right) x_j^{b_j(O)-\delta-1}\,dz.
$$
Using condition \eqref{eq:Cond_b_j_3} we see that $|I^{\eps}_i|\leq C\eps$, and so
\begin{equation}
\label{eq:I_eps_conv}
I^{\eps}_i\rightarrow 0,\quad\hbox{ as } \eps\downarrow 0,\quad\forall\, 1\leq i\leq n.
\end{equation}
Letting $\eps$ tend to zero in identity \eqref{eq:Weak_sol_var_eq_tilde_eps}, using property \eqref{eq:I_eps_conv} and Lemma \ref{eq:tilde_u_space}, we obtain that \eqref{eq:Weak_sol_var_eq_tilde} holds, and so $\widetilde u$ is a local weak solution to equation \eqref{eq:Problem_local_tilde_u}. This completes the proof.
\end{proof}

We can now give the proof of the main result

\begin{proof}[Proof of Theorem \ref{thm:Boundary_reg}]
Using the fact that $u=(w^T)^{-1}\widetilde u$ and that 
$$
\|\widetilde u\|_{L^2((0,T)\times\cB_R;d\widetilde\mu)}
=
\|u\|_{L^2((0,T)\times\cB_R;d\mu)},
$$
the estimates in the statement of Theorem \ref{thm:Boundary_reg} follow immediately provided we prove that for all $l\in\NN$, $0<r<R$, $0<t<T$, we have that
\begin{equation}
\label{eq:C_l_tilde_u}
\|\widetilde u\|_{C^l([t,T]\times\bar\cB_{R/4})} \leq C\|\widetilde u\|_{L^2((0,T)\times\cB_R;d\widetilde\mu)},
\end{equation}
where $C=C(\fb,l,L,t,T,R)$ is a positive constant. From \cite{Epstein_Mazzeo_2016}, for all $0<r<R$ and $0<t<T$, there is a positive constant $C=C(\fb,L,t,T,r,R)$ such that
\begin{equation}
\label{eq:L_infty_tilde_u}
\|\widetilde u\|_{L^{\infty}((t/2,T)\times\cB_r)} \leq C\|\widetilde u\|_{L^2((0,T)\times\cB_R;d\widetilde\mu)}.
\end{equation}
Let $\{\widetilde u_k\}_{k\in\NN}\in C^{\infty}([0,\infty)\times \bar\cB_{R/2})$ be the sequence of functions constructed in Lemma \ref{lem:Approx_smooth}, which we apply with $L$ replaced by $\widetilde L$ and $r=R/4$. Using the local a priori supremum estimate  
\cite[Theorem 1.1]{Pop_2013b} and inequality \eqref{eq:L_infty_tilde_u}, it follows that for all $\eps>0$ there is $k_{\eps}\in\NN$ such that
\begin{equation*}
\|\widetilde u_k\|_{C^l([t,T]\times\bar\cB_{R/4})} \leq C \|\widetilde u\|_{L^2((0,T)\times\bar\cB_R)}+\eps,\quad\forall\, k\geq k_{\eps}.
\end{equation*}
Letting $\eps$ tend to zero and applying the Arzel\'a-Ascoli Theorem together with property \eqref{eq:u_k_derivatives_conv}, we obtain estimate \eqref{eq:C_l_tilde_u}. This concludes the proof.
\end{proof}

We continue to apply the conjugation property of generalized Kimura operators to give the proofs of the boundary Harnack principles stated in Theorems \ref{thm:Carleson_Hopf_Oleinik} and \ref{thm:Holder_cont}.

\begin{proof}[Proof of Theorem \ref{thm:Carleson_Hopf_Oleinik}]
According to Lemma \ref{eq:tilde_u_solution}, the function $\widetilde u$
defined by \eqref{eq:tilde_u} is a non-negative, local weak solution to the equation $\widetilde u_t-\widetilde L\widetilde u = 0$ on $(t-4r^2,t+4r^2)\times B_{2r}(z)$. We recall that the generalized Kimura operator $\widetilde L$ in \eqref{eq:Conjugation_op} has positive weights, and so we can apply the Harnack inequality established in \cite[Theorem 1.2]{Epstein_Pop_2013b} and in \cite[Theorem 4.1]{Epstein_Mazzeo_2016} to conclude that there is a positive constant, $H=H(\fb, L, n,m)$, such that
\begin{align*}
\sup_{Q_r(t,z)} \widetilde u \leq H \inf_{Q^+_r(t,z)} \widetilde u
\qquad\hbox{ and }\qquad
\sup_{Q^-_r(t,z)} \widetilde u \leq H \inf_{Q_r(t,z)} \widetilde u.
\end{align*}
Recalling the relationship between the solutions $u$ and $\widetilde u$ in \eqref{eq:tilde_u} and the definition of the point $A_r(z)$ in \eqref{eq:A_r}, the Carleson-type estimate \eqref{eq:Carleson} and the Hopf-Oleinik-type estimate \eqref{eq:Hopf_Oleinik} are direct consequences of the first and second inequalities above, respectively. The supremum quotient bound \eqref{eq:Quotient_bounds_sup} follows from estimates \eqref{eq:Carleson} and \eqref{eq:Hopf_Oleinik} applied to $u_1$ and $u_2$, respectively, while the infimum quotient bound \eqref{eq:Quotient_bounds_inf} follows from estimates \eqref{eq:Carleson} and \eqref{eq:Hopf_Oleinik} applied to $u_2$ and $u_1$, respectively. This completes the proof.
\end{proof}

We conclude this section with the

\begin{proof}[Proof of Theorem \ref{thm:Holder_cont}]
Employing again the conjugation argument as in the proof of Theorem \ref{thm:Carleson_Hopf_Oleinik}, we have that $u_1/u_2 = \widetilde u_1/{\widetilde u_2}$. Because the operator $\widetilde L$ has positive weights, we can apply \cite[Corollary 4.1]{Epstein_Mazzeo_2016} to obtain that there is a positive constant, $\alpha=\alpha(\fb,L,n,m)$, such that $\widetilde u_i$ belongs to $C^{\alpha}_{\hbox{\tiny{WF}}}(\bar Q_r(z))$, for $i=1,2$. From here the conclusion of Theorem \ref{thm:Holder_cont} follows.
\end{proof}

\subsection{Boundary behavior of global weak solutions}
\label{sec:Boundary_behavior_global}
In this section to give the proofs of the main results stated in \S\ref{sec:Global_regularity}. We begin with the

\begin{proof}[Proof of Theorem \ref{thm:Global_regularity}]
It is clear that the conclusion of Theorem \ref{thm:Global_regularity} is an immediate consequence of Theorem \ref{thm:Boundary_reg}.
\end{proof}

To give the proof of Theorem \ref{thm:Quotient_bounds_global}, we first need to establish an elliptic-type Harnack inequality. For all $r>0$ and $t\in (r^2,T)$, we denote
\begin{equation}
\label{eq:P_r_t}
P_r(t) := (t-r^2,t)\times P.
\end{equation}

Analogously to the elliptic-type Harnack inequality for strictly elliptic operators, \cite[Theorem 1.3]{Fabes_Garofalo_Salsa_1986}, we next prove

\begin{lem}[Elliptic-type Harnack inequality]
\label{lem:Elliptic_Harnack}
Let $u$ be a global, non-negative, weak solution to the initial-value problem
\eqref{eq:Initial_value_problem}. For all $r\in (0,1)$, there is a positive
constant, $H=H(\fb,L,n,m,r)$, such that for all $4r^2<t<T-4r^2$ we have that
\begin{equation}
\label{eq:Elliptic_Harnack}
\sup_{P_r(t)} w^Tu \leq H \inf_{P_r(t)} w^Tu.
\end{equation}
\end{lem}

\begin{rmk}
  Notice that the main difference between the statements of Lemmas
  \ref{lem:Elliptic_Harnack} and \cite[Theorem 1.3]{Fabes_Garofalo_Salsa_1986}
  is that in inequality \eqref{eq:Elliptic_Harnack} the domain $P_r(t)$
  comprises also the boundary $(t-r^2,t)\times \partial P$, while in the
  statement of \cite[Theorem 1.3]{Fabes_Garofalo_Salsa_1986}, the domain
  $P_r(t)$ is replaced by $(t-r^2,t)\times \{p\in P:\hbox{dist}(p,\partial
  P)>r\}$. However, our estimate \eqref{eq:Elliptic_Harnack} is weighted by the
  function $w^T(z)$ defined in \eqref{eq:Weight_density_tangent}, while in
  \cite[Theorem 1.3]{Fabes_Garofalo_Salsa_1986} there is no weight used.
\end{rmk}

\begin{proof}[Proof of Lemma \ref{lem:Elliptic_Harnack}]
  The argument employed in the proof of Lemma \ref{eq:tilde_u_solution} yields
  that $\widetilde u$ defined by \eqref{eq:tilde_u} is a global, non-negative,
  weak solution to the equation $\widetilde u_t-\widetilde L\widetilde u = 0$
  on $(0,T)\times P$. Because the generalized Kimura operator $\widetilde L$ in
  \eqref{eq:Conjugation_op} has positive weights, we can apply a covering
  argument together with the Harnack inequality established in \cite[Theorem
  1.2]{Epstein_Pop_2013b} and in \cite[Theorem 4.1]{Epstein_Mazzeo_2016} to
  conclude that there is a positive constant, $H=H(\fb, L, n,m,r)$, such that
\begin{align}
\label{eq:Harnack_P_r_t_1}
\sup_{P_r(t)} \widetilde u &\leq H \inf_{P^+_r(t)} \widetilde u,\\
\label{eq:Harnack_P_r_t_2}
\sup_{P^-_r(t)} \widetilde u &\leq H \inf_{P_r(t)} \widetilde u,
\end{align}
where we recall the definition of the set $P_r(t)$ in \eqref{eq:P_r_t}, and we let $P^+_r(t):=(t+r^2, t+2r^2)\times P$ and $P^-_r(t):=(t-3r^2, t-2r^2)\times P$. The maximum principle 
\cite[Theorem]{Epstein_Mazzeo_annmathstudies} gives us that 
\begin{equation}
\label{eq:Max_princ_P_r_t}
\widetilde u \leq C \sup_{P^-_r(t)} \widetilde u,\quad\forall\, (t',z')\in (t-2r^2,t+2r^2]\times P.
\end{equation}
where $C=C(L,n,m,T)$ is a positive constant. Thus, we obtain that
\begin{align*}
\sup_{P_r(t)} \widetilde u 
&\leq H \inf_{P^+_r(t)} \widetilde u 
\quad\hbox{ (by \eqref{eq:Harnack_P_r_t_1})}\\
&\leq CH \sup_{P^-_r(t)} \widetilde u
\quad\hbox{ (by \eqref{eq:Max_princ_P_r_t})}\\
&\leq CH^2 \inf_{P_r(t)} \widetilde u
\quad\hbox{ (by \eqref{eq:Harnack_P_r_t_2})},
\end{align*}
which gives us the elliptic-type Harnack inequality \eqref{eq:Elliptic_Harnack}.
\end{proof}

Lemma \ref{lem:Elliptic_Harnack} allows us now to give the

\begin{proof}[Proof of Theorem \ref{thm:Quotient_bounds_global}]
The elliptic-type Harnack inequality \eqref{eq:Elliptic_Harnack} gives us that
\begin{align*}
u_1(s,q)\leq H u_1(t, A_r(p))
\quad\hbox{ and }\quad
u_2(t, A_r(p)) \leq H u_2(s,q),
\quad\forall\, (s,q)\in Q_r(t,p),
\end{align*}
which implies that
$$
\sup_{Q_r(t,p)} \frac{u_1}{u_2} \leq H^2 \frac{u_1(t,A_r(p))}{u_2(t,A_r(p))}.
$$
Reversing the roles of $u_1$ and $u_2$, it follows that
$$
\inf_{Q_r(t,p)} \frac{u_1}{u_2} \geq H^{-2} \frac{u_1(t,A_r(p))}{u_2(t,A_r(p))},
$$
and so the preceding two inequalities yield property \eqref{eq:Quotient_bounds_global}.
\end{proof}

Since $\widetilde u=w_T u,$ we see that
\begin{equation}
  \|w_Tu(t,\cdot)\|_{C^{0,\gamma}_{\WF}(P)}=
\|\widetilde u(t,\cdot)\|_{C^{0,\gamma}_{\WF}(P)}
\end{equation}
is finite.  Let $C^{0,\gamma}_{\WF,D}(P)$ denote the subspace of
$C^{0,\gamma}_{\WF}(P)$ for which the norm
$\|w_Tf\|_{C^{0,\gamma}_{\WF}(P)}$ is finite, with this norm.  We can
apply the proof of Corollary 5.1 in~\cite{Epstein_Mazzeo_2016} to
prove the following corollary.
\begin{cor} Assume that $\cL$ is a second-order differential operator defined
  on a compact manifold with corners $P$ such that when written in a
  local system of coordinates it takes the form of the operator $L$
  defined in \eqref{eq:Operator} and it satisfies Assumption
  \ref{assump:Coeff}.  If $\overline{\cL_0}$ is the graph closure of
  $\cL$ acting on $C^0_{D}(P),$ which is $C^{\infty}_c(P\setminus
  \partial^T P),$ with respect to the weighted sup-norm:
  $$\|f\|_{0,D}=\sup_{x\in P}|w_T(x)f(x)|,$$
  then, for any $0<\gamma<1,$ and
  $\mu$ with $\Re\mu >0,$ the resolvent operator
  $(\mu\Id-\overline{\cL_0})^{-1}:C^0_{D}(P)\to C^{0,\gamma}_{\WF,D}(P)$ is bounded,
  and is therefore a compact operator on $C^0_{D}(P).$
\end{cor} 
\begin{proof} We only need to show that bounded subsets of $C^{0,\gamma}_{\WF,D}(P)$ are compact subsets of $C^0_D(P),$ but this follows easily from the Arzela-Ascoli theorem. 
  \end{proof}
\appendix

\section{Appendix}
\label{sec:Appendix}
In this appendix, we collect several auxiliary results which we use throughout the article.

\begin{lem}
\label{lem:Bad_first_derivative}
Let $d\mu$ and $d\mu_{\fe_i}$ be the measures defined in \eqref{eq:Weight} and \eqref{eq:Weight_for_higher_order_Sobolev_spaces}, respectively, where we assume that the weights $\{b_i:1\leq i\leq n\}$ satisfy condition \eqref{eq:Cleanness} and $n_0+1 \leq i \leq n$, where $n_0$ is the integer defined in \eqref{eq:n_0}. Let $0<r<R<1$, $\varphi:\bar S_{n,m}\rightarrow\RR$ be a smooth function with compact support in $\bar \cB_r$, and $u\in H^1(\cB_R; d\mu_{\fe_i})$. Then $\varphi u$ belongs to $L^2(\cB_R;d\mu)$ and there is a positive constant, $C=C(b,m,n)$, such that 
\begin{equation}
\label{eq:Bad_first_derivative}
\begin{aligned}
\int_{\cB_R}  |u|^2 \varphi^2\,d\mu
&\leq 
\int_{\cB_R}\left(\sum_{j=1}^nx_j|D^{\fe_j}_xu|^2 +\sum_{l=1}^m|D^{\ff_l}_yu|^2\right)\varphi^2\,d\mu_{\fe_i}\\
&\quad + C\int_{\cB_R}\left(\varphi^2+|\nabla\varphi|^2\right)|u|^2\, d\mu_{\fe_i}.
\end{aligned}
\end{equation}
\end{lem}

\begin{proof}
Because the space of functions 
$C^{\infty}_c(\bar\cB_R\backslash \partial^T \cB_R)$ is dense in $H^1(\cB_R;d\mu)$, it follows that it is sufficient to prove inequality \eqref{eq:Bad_first_derivative} under the assumption that $u$ belongs to 
$C^{\infty}_c(\bar\cB_R\backslash \partial^T\cB_R)$, which we assume for the rest of the proof. Integration by parts and the fact that the measure $d\mu$ defined in \eqref{eq:Weight} has the property that the exponent $b_i$ does not depend on the $x_i$-variable give us that
\begin{equation}
\label{eq:Bad_term}
\begin{aligned}
\int_{\cB_R}  |u|^2 \varphi^2\,d\mu
&= \int_{\cB_R}  |u|^2 \varphi^2 \frac{1}{b_i(z)}\partial_{x_i}(x_i^{b_i(z)}) x_i^{-b_i(z)+1} \,d\mu
-\int_{\cB_R}  |u|^2 \varphi^2 \frac{1}{b_i(z)}(\partial_{x_i} b_i \ln x_i) x_i\,d\mu \\
& = - \int_{\cB_R}  \frac{2}{b_i(z)}\left(\varphi D^{\fe_i}_x \varphi|u|^2 + \varphi^2 u D^{\fe_i}_xu\right)x_i\,d\mu\\
&\quad 
-\int_{\cB_R}  \frac{1}{b_i(z)} |u|^2 \varphi^2 \sum_{k=1}^n (\partial_{x_i} b_k\ln x_k) x_i\,d\mu.
\end{aligned}
\end{equation}
Because  $b_i\geq \beta_0>0$ on $\{x_i=0\}\cap\partial\cB_1$ by condition \eqref{eq:Cleanness}, we can assume without loss of generality that $b_i \geq \beta_0/2$ on $\bar \cB_1$. From definition \eqref{eq:Weight_for_higher_order_Sobolev_spaces} of the weight 
$d\mu_{\fe_i}$, we have that
\begin{align}
\label{eq:Bad_term_1}
\int_{\cB_R}  \int_{\cB_R}  \frac{2}{b_i(z)}\varphi |D^{\fe_i}_x \varphi||u|^2 x_i\,d\mu
\leq 
\frac{4}{\beta_0}\int_{\cB_R}\left(\varphi^2 + |\nabla\varphi|^2\right)|u|^2\,d\mu_{\fe_i},
\end{align}
and, for all $\eps>0$, there is a positive constant, $C=C(b,\eps,m,n)$, such that
\begin{align}
\label{eq:Bad_term_2}
\int_{\cB_R} \frac{2}{b_i(z)} \varphi^2 |u| |D^{\fe_i}_x u| x_i\,d\mu
\leq 
\eps \int_{\cB_R}|u|^2 \varphi^2\,d\mu
\quad + C\int_{\cB_R}x_i|D^{\fe_i}_x u|^2\varphi^2\, d\mu_{\fe_i}.
\end{align}
Applying \cite[Lemma B.3]{Epstein_Mazzeo_annmathstudies}, 
\footnote{We recall that the proof of \cite[Lemma B.3]{Epstein_Mazzeo_2016} requires that the weights $b_k\restrictedto_{\{x_k=0\}}$ are positive, for all $1\leq k\leq n$, but it is easy to see that the same proof also holds in the case when there are logarithmic singularities $\ln x_k$ only corresponding to weights $b_k$ such that $b_k\restrictedto_{\{x_k=0\}}$ is positive, i.e. $n_0+1\leq k\leq n$ by \eqref{eq:n_0}. Note that this is the case in the present setting since $\partial_{x_i} b_k = 0$, for all $1\leq k\leq n_0$.}
we can bound the last term on the right-hand side in the preceding inequality by
\begin{equation}
\label{eq:Bad_term_3}
\begin{aligned}
\int_{\cB_R}  \frac{1}{b_i(z)} |u|^2 \varphi^2 \sum_{k=1}^n |\partial_{x_i} b_k||\ln x_k|x_i\,d\mu
&\leq 
\eps\int_{\cB_R}\left(\sum_{j=1}^nx_j|D^{\fe_j}_xu|^2 +\sum_{l=1}^m|D^{\ff_l}_y u|^2\right)\varphi^2\,d\mu_{\fe_i}\\
&\quad + C\int_{\cB_R}\left(\varphi^2+|\nabla\varphi|^2\right)|u|^2\, d\mu_{\fe_i}.
\end{aligned}
\end{equation}
for all $\eps>0$, where $C=C(b,\eps,m,n)$ is a positive constant. Inequality \eqref{eq:Bad_first_derivative} follows now by choosing $\eps$ small enough and using identity \eqref{eq:Bad_term} together with inequalities \eqref{eq:Bad_term_1}, \eqref{eq:Bad_term_2}, and \eqref{eq:Bad_term_3}. This completes the proof.
\end{proof}

To state the next auxiliary result, we make use of the commutator identity:
\begin{equation}
\label{eq:Commutator}
\begin{aligned}
{[}L,\varphi {]} u 
&
:= uL\varphi + \sum_{i=1}^n 2x_i\bar a_{ii}u_{x_i}\varphi_{x_i} + \sum_{i,j=1}^n x_ix_ja_{ij}\left(u_{x_i}\varphi_{x_j}+u_{x_j}\varphi_{x_i}\right)\\
&
\quad
+\sum_{i=1}^n\sum_{l=1}^m x_ic_{il}\left(u_{x_i} \varphi_{y_l} + u_{y_l}\varphi_{x_i}\right)
+\sum_{l,k=1}^m d_{lk}\left(u_{y_l} \varphi_{y_k} + u_{y_k}\varphi_{y_l}\right),
\end{aligned}
\end{equation}
for all $\varphi, u \in C^1(\cB_1)$, where the operator $L$ is defined in \eqref{eq:Operator}. We have

\begin{lem}
\label{eq:Cutoff_local_solution}
Assume that the operator $L$ satisfies Assumption \ref{assump:Coeff}. Let $0<r<R<1$, $T>0$, $f\in L^2(\cB_R; d\mu)$, and let $u$ be a local weak solution to the initial-value problem \eqref{eq:Initial_value_problem_local}. Let $\varphi:\bar S_{n,m}\rightarrow\RR$ be a smooth function with compact support in $\bar\cB_r$. Then the function $\bar u = \varphi u$ is a weak solution to 
\begin{equation}
\label{eq:Problem_cutoff}
\begin{aligned}
\left\{\begin{array}{rl}
\bar u_t-L\bar u=\bar g & \hbox{ on } (0, T)\times \cB_R,\\ 
\bar u(0) = \bar f& \hbox{ on } \cB_R, 
\end{array} \right.
\end{aligned}
\end{equation}
where $\bar f:= \varphi f$, $\bar g= -[L,\varphi] u$, and $\bar g$ belongs to $L^2((0,T);L^2(\cB_R;d\mu))$.
\end{lem}

\begin{proof}
Because $u$ is a weak solution to the equation \eqref{eq:Initial_value_problem_local}, it follows that $u$ belongs to the space of functions $L^2((0,T); H^1(\cB_R;d\mu))$, and so from \eqref{eq:Commutator} and the definition of the Sobolev space $H^1(\cB_R;d\mu)$, we have that $\bar g$ belongs to 
$L^2((0,T);L^2(\cB_R;d\mu))$. The initial condition $\bar u(0) = \bar f$ is clearly satisfied by \eqref{eq:Weak_sol_initial_cond} and the definition of local weak solutions in \S \ref{sec:Weak_sol}. Thus, it only remains to prove that, for all test functions $v\in \cF((0,T)\times \bar \cB_R)$, we have that
\begin{equation}
\label{eq:Weak_sol_var_eq_cutoff}
\int_0^T\left(\frac{d\bar u(t)}{dt}, v(t)\right)_{L^2(\cB_R;d\mu)}\, dt+\int_0^T Q(\bar u(t),v(t))\, dt= \int_0^T \left(\bar g(t), v(t)\right)_{L^2(\cB_R;d\mu)}\, dt. 
\end{equation}
We clearly have that
$$
\int_0^T\left(\frac{d\bar u(t)}{dt}, v(t)\right)_{L^2(\cB_R;d\mu)}\, dt
=
\int_0^T\left(\frac{d u(t)}{dt}, \varphi v(t)\right)_{L^2(\cB_R;d\mu)}\, dt,
$$
and we will prove that
\begin{equation}
\label{eq:Q_cutoff}
\int_0^T Q(\bar u(t),v(t))\, dt
=
\int_0^T Q(u(t),\varphi v(t))\, dt + \int_0^T \left(\bar g(t), v(t)\right)_{L^2(\cB_R;d\mu)}\, dt.
\end{equation}
The preceding two identities together with the fact that $u$ is a weak solution of equation \eqref{eq:Initial_value_problem_local} and that $\bar g:=[L,\varphi]u$ imply \eqref{eq:Weak_sol_var_eq_cutoff}. Thus, we only need to establish that identity \eqref{eq:Q_cutoff} holds.

Because the space of functions 
$C^{\infty}_c([0,T]\times(\bar \cB_R\backslash\partial^T\cB_R))$ is dense in $L^2((0,T);H^1(\cB_R;d\mu))$, we can assume without loss of generality that $u$ belongs to 
$C^{\infty}_c([0,T]\times(\bar \cB_R\backslash\partial^T \cB_R))$. Usual integration by parts gives us that
\footnote{We suppress the time variable for clarity.}
\begin{align*}
Q(\bar u,v)) &= -(L(\varphi u), v)_{L^2(\cB_R;d\mu)}\\
& = -(\varphi Lu, v)_{L^2(\cB_R;d\mu)} -([L, \varphi] u, v)_{L^2(\cB_R;d\mu)} 
\quad\hbox{ (because $L(\varphi u) = \varphi Lu + [L,\varphi] u$)}\\
& = -(Lu, \varphi v)_{L^2(\cB_R;d\mu)} +(\bar g, v)_{L^2(\cB_R;d\mu)}
\quad\hbox{ (because $\bar g:=-[L,\varphi]u$)}\\
& = Q(u, \varphi v) +(\bar g, v)_{L^2(\cB_R;d\mu)},
\end{align*}
from where identity \eqref{eq:Q_cutoff} follows. This completes the proof.
\end{proof}

We next prove:

\begin{lem}
\label{lem:Sup_est_smooth}
Let $d\mu$ be the measure defined in \eqref{eq:Weight} such that the coefficients $\{b_i:1\leq i\leq n\}$ satisfy condition 
\eqref{eq:Cleanness} holds. Let $u:\cB_R\rightarrow\RR$ be a smooth function. Then for all $0<r<r_0<R$ and $\fb\in\NN^m$, there are positive constants, $C=C(b,\fb,m,n,r,r_0,R)$, $p=p(b,m,n)>2$, and $q=q(b,m,n)<2$, such that for all $z\in\cB_r$ we have
\begin{equation}
\label{eq:Sup_est_deriv_slice}
|D^{\fe}_xD^{\fb}_yu(z)| 
\leq C \cN \cW_{r_0}(z),
\end{equation}
where the norm $\cN$ and the set of multi-indices $\cD$ 
\begin{align}
\label{eq:Definition_cN}
\cN&:=\sum_{(\fa,\fc)\in\cD} \|D^{\fa+\fe}_xD^{\fc+\fb}_y u\|_{H^1(\cB_R;d\mu_{\fa+\fe})},\\
\label{eq:Definition_cD}
\cD&:=\{(\fa,\fc)\in\NN^n\times\NN^m:\, \fa_i, \fc_l\in \{0,1\}, \hbox{ for all } 1\leq i\leq n \hbox{ and } 1\leq l\leq m\},
\end{align}
the factor $\cW_{r_0}(z)$ are defined in \eqref{eq:Definition_cW} and the rectangle $\cR_{r_0}(z)$ is defined in 
\eqref{eq:Definition_cR}.
\end{lem}

\begin{proof}
Let $\varphi:\RR\rightarrow [0,1]$ be a smooth cut-off function such that $\varphi(s)=1$ for all $|s|\leq r$, and $\varphi(s)=0$ for all $|s|\geq r_0$. Let $v:=D^{\fe+(0,\fb)}u$. For all $z\in \cB_r$, we have by integration by parts
\begin{equation*}
\begin{aligned}
v(s,z) &= -\int_{x_1}^{\infty} d\xi_1D^{\fe_1}_{\xi}(\varphi(\xi_1)v(\xi_1,x_2,\ldots,x_n,y))\\
&= \int_{x_2}^{\infty}d\xi_2\int_{x_1}^{\infty}d\xi_1  D^{\fe_2}_{\xi} (\varphi(\xi_2)D^{\fe_1}_{\xi}(\varphi(\xi_1)v(\xi_1,\xi_2,x_3,\ldots,x_n,y)))\\
&= (-1)^{n+m} \int_{y_m}^{\infty}d\rho_m\ldots \int_{y_1}^{\infty}d\rho_1 \int_{x_n}^{\infty}d\xi_n\ldots\int_{x_1}^{\infty}d\xi_1\\
&\quad\quad
D^{\ff_m}_{\rho} (\varphi(\rho)\ldots D^{\ff_1}_{\rho} \varphi(\rho_1) D^{\fe_n}_{\xi}(\varphi(\xi_n)\ldots D^{\fe_1}_{\xi}(\varphi(\xi_1)v(\xi,\rho))))).
\end{aligned}
\end{equation*}
More compactly, using definitions \eqref{eq:Definition_cD} and \eqref{eq:Definition_cR}, we can write the preceding identity as
\begin{align*}
v(z) &= (-1)^{n+m} \sum_{(\fa,\fc)\in\cD}\int_{\cR_{r_0}(z)} D^{\fa}_{\xi}D^{\fc}_{\rho}v(\xi,\rho)
D^{\fe-\fa}_{\xi}D^{\fc}_{\rho}(\varphi(\rho_1)\ldots\varphi(\rho_m)\varphi(\xi_1)\ldots\varphi(\xi_n))\, d\xi d\rho,
\end{align*}
and so there is a positive constant, $C=C(r,r_0)$, such that
\begin{equation}
\label{eq:Sup_est_int}
|v(z)| \leq C\sum_{(\fa,\fc)\in\cD} \int_{\cR_{r_0}(z)} |D^{\fa}_{\xi}D^{\fc}_{\rho}v(\xi,\rho)|\, d\xi d\rho.
\end{equation}
For all $(\fa,\fc)\in \cD$, notice that the weight function $d\mu_{\fa+\fe}$ defined using \eqref{eq:Weight_for_higher_order_Sobolev_spaces} satisfies the hypotheses of the Sobolev inequality \cite[Theorem 3.2]{Epstein_Mazzeo_2016}. This yields that there are positive constants, $C$ and $p_{\fa}=p(\fa,b,m,n)>2$, such that 
\begin{equation*}
\left(\int_{\cB_{r_0}} |D^{\fa}_xD^{\fc}_yv|^{p_{\fa}}\, d\mu_{\fa+\fe}\right)^{1/p_{\fa}} \leq C \cN,
\end{equation*}
where the norm $\cN$ is defined by \eqref{eq:Definition_cN}. Let $p:=\min\{p_{\fa}:\fa\in\NN^n, |\fa|=1\}$. Then the preceding inequality implies that
\begin{equation*}
\left(\int_{\cB_{r_0}} |D^{\fa}_{x}D^{\fc}_{y}v|^p\, d\mu_{\fa+\fe}\right)^{1/p} \leq C \cN,
\end{equation*}
which together with \eqref{eq:Sup_est_int} and H\"older's inequality give us
\begin{align*}
|v(z)| &\leq C\cN 
\left(\int_{\cR_{r_0}(z)}
\prod_{i=1}^{n_0} \xi_i^{-\fa_i \frac{q}{p}} \, d\xi_i
\prod_{j=n_0+1}^{n} \xi_j^{-(b_j(\xi,\rho)+\fa_j-1) \frac{q}{p}} \, d\xi_j
\prod_{l=1}^md\rho_l
\right)^{\frac{1}{q}},
\end{align*}
where we denote by $q$ the conjugate exponent of $p$.
The right-hand side above attains its largest value when $\fa_k=1$, for all $1\leq k\leq n$, which implies \eqref{eq:Sup_est_deriv_slice}. This completes the proof.
\end{proof}

%
%

\def\cprime{$'$} \def\cprime{$'$}
  \def\polhk#1{\setbox0=\hbox{#1}{\ooalign{\hidewidth
  \lower1.5ex\hbox{`}\hidewidth\crcr\unhbox0}}} \def\cprime{$'$}
  \def\cprime{$'$} \def\cprime{$'$} \def\cprime{$'$} \def\cprime{$'$}
  \def\lfhook#1{\setbox0=\hbox{#1}{\ooalign{\hidewidth
  \lower1.5ex\hbox{'}\hidewidth\crcr\unhbox0}}} \def\cprime{$'$}
  \def\cprime{$'$} \def\cprime{$'$} \def\cprime{$'$} \def\cprime{$'$}
  \def\cprime{$'$} \def\cprime{$'$} \def\cprime{$'$} \def\cprime{$'$}
  \def\cprime{$'$}
\providecommand{\bysame}{\leavevmode\hbox to3em{\hrulefill}\thinspace}
\providecommand{\MR}{\relax\ifhmode\unskip\space\fi MR }
\providecommand{\MRhref}[2]{%
  \href{http://www.ams.org/mathscinet-getitem?mr=#1}{#2}
}
\providecommand{\href}[2]{#2}

\end{document}